\documentclass[11pt]{amsart}
\usepackage{amsfonts}
\usepackage{amsmath}
\usepackage{amsthm}
\usepackage{url}
\usepackage[all]{xy}
\usepackage{latexsym}
\usepackage{amssymb}
\usepackage[cp850]{inputenc}
\usepackage{amsfonts}
\usepackage{graphicx}

\newtheorem{sat}{Theorem}[section]		\newtheorem{lem}[sat]{Lemma}
\newtheorem{kor}[sat]{Corollary}			\newtheorem{prop}[sat]{Proposition}

\newtheorem*{defi*}{Definition}			\newtheorem*{bei*}{Example}
\newtheorem*{sat*}{Theorem}				\newtheorem*{kor*}{Corollary}
\newtheorem*{rmk*}{Remark}					
\newtheorem{fact}{Fact}

\let\ssection=\section
\renewcommand{\section}{\setcounter{equation}{0}\ssection}

\newtheorem*{namedtheorem}{\theoremname}
\newcommand{\theoremname}{testing}
\newenvironment{named}[1]{\renewcommand{\theoremname}{#1}\begin{namedtheorem}}{\end{namedtheorem}}

\theoremstyle{remark}
\newtheorem*{bem}{Remark}

\newtheorem*{namedtheoremr}{\theoremnamer}
\newcommand{\theoremnamer}{testing}

			\newcommand{\BH}{\mathbb H}
\newcommand{\BR}{\mathbb R}			
\newcommand{\BN}{\mathbb N}			
\newcommand{\BS}{\mathbb S}			\newcommand{\BZ}{\mathbb Z}

\newcommand{\CA}{\mathcal A}		
\newcommand{\CC}{\mathcal C}		
		\newcommand{\CF}{\mathcal F}
\newcommand{\CG}{\mathcal G}		
\newcommand{\CI}{\mathcal I}		
		\newcommand{\CL}{\mathcal L}
\newcommand{\CM}{\mathcal M}		
\newcommand{\CO}{\mathcal O}		\newcommand{\CP}{\mathcal P}
		
\newcommand{\CS}{\mathcal S}		
		\newcommand{\CV}{\mathcal V}
\newcommand{\CW}{\mathcal W}		
		\newcommand{\CZ}{\mathcal Z}

\newcommand{\D}{\partial}

\DeclareMathOperator{\SL}{SL}		
\DeclareMathOperator{\Id}{Id}		
\DeclareMathOperator{\vol}{vol}		

\DeclareMathOperator{\Map}{Map}

\newcommand{\comment}[1]{}

\DeclareMathOperator{\Thu}{Thu}

\DeclareMathOperator{\fills}{\prec_{\mathrm\tiny fills}}

\begin{document}

\title[]{Counting curves in hyperbolic surfaces}
\author{Viveka Erlandsson}
\address{Aalto Science Institute, Aalto University}
\email{viveca.erlandsson@aalto.fi}
\author{Juan Souto}
\address{IRMAR, Universit\'e de Rennes 1}
\email{juan.souto@univ-rennes1.fr}
\thanks{\today}

\begin{abstract}
Let $\Sigma$ be a hyperbolic surface. We study the set of curves on $\Sigma$ of a given type, i.e. in the mapping class group orbit of some fixed but otherwise arbitrary $\gamma_0$. For example, in the particular case that $\Sigma$ is a once-punctured torus, we prove that the cardinality of the set of curves of type $\gamma_0$ and of at most length $L$ is asymptotic to $L^2$ times a constant.
\end{abstract}

\maketitle

\section{}

Throughout this paper we let $\Sigma$ be a complete hyperbolic surface of finite area, with genus $g$ and $r$ punctures, and distinct from a thrice punctured sphere. By an {\em immersed multicurve}, or simply {\em multicurve}, in $\Sigma$ we will mean an immersed compact 1-dimensional submanifold of $\Sigma$ each of whose components represents (the conjugacy class of) a primitive non-peripheral element in $\pi_1(\Sigma)$. Two multicurves $\gamma,\gamma'$ are of the same {\em type} if they belong to the same mapping class orbit, meaning that there is a diffeomorphism $\phi$ of $\Sigma$ such that $\gamma$ and $\phi(\gamma')$ are isotopic as immersed submanifolds. In general, isotopic geodesics are considered to be equivalent. For instance, every multicurve $\gamma$ is isotopic to a geodesic multicurve and the length $\ell_\Sigma(\gamma)$ is the length of the latter.

In this paper we study the set $\CS_{\gamma_0}=\Map(\Sigma)\cdot\gamma_0$ of (isotopy classes of) multicurves of some given type $\gamma_0$. More precisely, we are interested in the behavior, when $L$ tends to infinity, of the number $\vert\{\gamma\in\CS_{\gamma_0}\vert\ell_\Sigma(\gamma)\le L\}\vert$ of multicurves in $\Sigma$ of type $\gamma_0$ and of at most length $L$. Since this number grows coarsely like a polynomial of degree $6g-6+2r$ (see \cite{Rees} for the case that $\gamma_0$ is simple and \cite{Sapir1,Sapir2} or Corollary \ref{poly-growth} below for the general case), the perhaps most grappling question is whether the limit 
\begin{equation}\label{eq1}
\lim_{L\to\infty}\frac{\vert\{\gamma\in\CS_{\gamma_0}\vert\ell_\Sigma(\gamma)\le L\}\vert}{L^{6g-6+2r}}
\end{equation}
exists. Our main result is that it does if $\Sigma$ is a once-punctured torus:

\begin{sat}\label{sat2}
Let $\Sigma$ be a complete hyperbolic surface of finite volume homeomorphic to a once punctured torus and let $\gamma_0\subset\Sigma$ be a multicurve. The limit \eqref{eq1} exists and moreover we have
$$\lim_{L\to\infty}\frac{\vert\{\gamma\in\CS_{\gamma_0}\vert\ell_\Sigma(\gamma)\le L\}\vert}{L^2}=C_{\gamma_0}\cdot\mu_{\Thu}(\{\lambda\in\CM\CL(\Sigma)\vert\ell_\Sigma(\lambda)\le 1\})$$
where $\mu_{\Thu}$ is the Thurston measure on the space of measured laminations $\CM\CL(\Sigma)$ and $C_{\gamma_0}>0$ depends only on $\gamma_0$.
\end{sat}

In the case of simple multicurves Theorem \ref{sat2} is due to McShane-Rivin \cite{Greg-Rivin}. Also, for simple multicurves, Mirzakhani \cite{Maryam} proved that the limit \eqref{eq1} exists for all $g$ and $r$. Building on the work of Mirzakhani, Rivin \cite{Rivin} established the existence of the limit \eqref{eq1} for multicurves with a single self-intersection. 

\begin{bem}
Recently, and independently of our work, Mirzakhani \cite{Maryam-new} has established the existence of \eqref{eq1} in complete generality. Her argument and ours are different in nature and in some sense complementary. See the remarks following the statement of Corollary \ref{kor-ratio} in this introduction for more on the relation between Mirzakhani's result and ours.
\end{bem}

Still in the setting of simple multicurves, the case of the torus is much more treatable than the general one because any two simple multicurves in the torus are of the same type as long as they have the same number of components. This means that, in the case of the torus, counting simple multicurves of some fixed type basically reduces to counting all simple multicurves, a much simpler problem. In fact, if $\Sigma$ is an arbitrary hyperbolic surface of finite area with genus $g$ and $r$ punctures, if $\CS$ is the set of all multicurves in $\Sigma$, and if we set 
$$c_\Sigma=\mu_{\Thu}(\{\lambda\in\CM\CL(\Sigma)\vert\ell_\Sigma(\lambda)\le 1\})$$
then
\begin{align*}
\lim_{L\to\infty}\frac{\vert\{\gamma\in\CS\vert\ell_\Sigma(\gamma)\le L,\ \iota(\gamma,\gamma)=0\}\vert}{L^{6g-6+2r}}
&=c_\Sigma\\
\lim_{L\to\infty}\frac{\vert\{\gamma\in\CS\vert\ell_\Sigma(\gamma)\le L,\ \iota(\gamma,\gamma)=1\}\vert}{L^{6g-6+2r}}
&=3(2g-2+r)\cdot c_\Sigma\\
\lim_{L\to\infty}\frac{\vert\{\gamma\in\CS\vert\ell_\Sigma(\gamma)\le L,\ \iota(\gamma,\gamma)=2\}\vert}{L^{6g-6+2r}}
&=\frac 92\big((2g+r)(2g+r-3)+2\big)\cdot c_\Sigma.
\end{align*}
See Proposition 3.1 in \cite{Maryam} for the first limit and Corollary \ref{kor-small-k} for the other two.
\medskip

As was the case in \cite{Maryam}, the basic strategy of this paper is to translate the problem of the existence of the limit \eqref{eq1} to the existence of a limit of a suitable family of measures. More concretely, we will consider, for $\gamma_0$ and $\CS_{\gamma_0}$ as above and for each $L$, the measure
$$\nu_{\gamma_0}^L=\frac 1{L^{6g-6+2r}}\sum_{\gamma\in\CS_{\gamma_0}}\delta_{\frac 1L\gamma}$$
on the space $\CC(\Sigma)$ of geodesic currents on $\Sigma$. Here, $\delta_{\frac 1L\gamma}$ is the Dirac measure centred in the current $\frac 1L\gamma$. We will study these measures when $L\to\infty$ and prove, for instance, that they can only accumulate to multiples of the Thurston measure $\mu_{\Thu}$ on the space $\CM\CL(\Sigma)$ of measure laminations:

\begin{named}{Proposition \ref{prop-sublimit}}
Any sequence $(L_n)_n$ of positive numbers with $L_n\to\infty$ has a subsequence $(L_{n_i})_i$ such that the measures $(\nu_{\gamma_0}^{L_{n_i}})_i$ converge in the weak-*-topology to the measure $\alpha\cdot\mu_{\Thu}$ on $\CM\CL(\Sigma)\subset\CC(\Sigma)$ for some $\alpha>0$.
\end{named}

The point of considering limits of these measures is that actual existence of the limit of the measures $\nu_{\gamma_0}^L$ implies (is equivalent to) the existence of the limit \eqref{eq1}. Before making this statement precise, recall that there is a (filling) current associated to the hyperbolic metric, the Liouville current $\lambda_\Sigma\in\CC(\Sigma)$, satisfying
\begin{equation}\label{eq-liouville}
\iota(\lambda_\Sigma,\gamma)=\ell_\Sigma(\gamma)
\end{equation}
for every curve $\gamma$. Here $\iota(\cdot,\cdot)$ is the intersection form on the space of currents and a current $\lambda$ is filling if every geodesic in $\Sigma$ is transversally intersected by some geodesic in the support of $\lambda$.

In light of \eqref{eq-liouville}, we can consider the limit \eqref{eq1} as a special case of the limit
\begin{equation}\label{eq101}
\lim_{L\to\infty}\frac{\vert\{\gamma\in\CS_{\gamma_0}\vert\iota(\lambda_0,\gamma)\le L\}\vert}{L^{6g-6+2r}}
\end{equation}
where $\lambda_0\in\CC(\Sigma)$ is filling. We prove:

\begin{named}{Proposition \ref{prop-reduction}}
Let $(L_n)_n$ be a sequence with $\lim_{n\to\infty}\nu_{\gamma_0}^{L_n}=\alpha\cdot\mu_{\Thu}$ for some $\alpha\in\BR_+$. Then
$$\lim_{n\to\infty}\frac{\vert\{\gamma\in\CS_{\gamma_0}\vert\iota(\lambda_0,\gamma)\le L_n\}\vert}{{L_n}^{6g-6+2r}}=\alpha\cdot\mu_{\Thu}(\{\lambda\in\CM\CL(\Sigma)\vert\iota(\lambda_0,\lambda)\le 1\})$$
for every filling current $\lambda_0\in\CC(\Sigma)$.
\end{named}

As a direct consequence we will get that the existence or non-existence of the limit \eqref{eq101} does not depend on the concrete (filling) current $\lambda_0$. In fact, we get something better:

\begin{named}{Corollary \ref{kor-ratio}}
Let $\Sigma$ be a hyperbolic surface of finite area, and let $\lambda_1,\lambda_2\in\CC(\Sigma)$ be filling currents. Then we have
$$\lim_{L\to\infty}\frac{\vert\{\gamma\in\CS_{\gamma_0}\vert\iota(\lambda_1,\gamma)\le L\}\vert}{\vert\{\gamma\in\CS_{\gamma_0}\vert\iota(\lambda_2,\gamma)\le L\}\vert}=\frac{\mu_{\Thu}(\{\lambda\in\CM\CL(\Sigma)\vert\iota(\lambda_1,\lambda)\le 1\})}{\mu_{\Thu}(\{\lambda\in\CM\CL(\Sigma)\vert\iota(\lambda_2,\lambda)\le 1\})}$$
for every multicurve $\gamma_0$ in $\Sigma$. Here $\mu_{\Thu}$ is as always the Thurston measure on the space of measured laminations $\CM\CL(\Sigma)$.
\end{named}

\begin{bem}
In view of Corollary \ref{kor-ratio}, it follows from Mirzakhani's result \cite{Maryam-new} on the existence of limit \eqref{eq1} that the limit \eqref{eq101} also exists for any possible filling current for any arbitrary hyperbolic surface of finite type. For instance, it follows that the analogue of the limit \eqref{eq1} also exists if we measure lengths with respect to an arbitrary metric of negative curvature \cite{Otal}, or with respect to a singular flat structure \cite{Chris-moon-kasra}. All this might be worth noting because Mirzakhani's arguments, using trace relations, may be hard to apply directly in these situations. This is what we meant when we claimed that the results in this paper and in \cite{Maryam-new} are to some extent complementary.
\end{bem}

\begin{bem}
As we will explain in a short digression at the end of section \ref{subsec-counting}, it follows from Corollary \ref{kor-ratio} and the work of Mirzakhani \cite{Maryam-new} that more general limits of the form \eqref{eq101} exist, where one replaces the multicurve $\gamma_0$ by a current $\alpha$.
\end{bem}

As we have seen, to establish the existence or non-existence of either limit \eqref{eq1} and \eqref{eq101} what we have to figure out is whether the measures $\nu_{\gamma_0}^L$ converge. Our strategy is to relate these measures to some other measures which are supported from the very beginning on the space of measured laminations. To do so we need to establish a relationship between the multicurves in $\CS_{\gamma_0}$, which are in general non-simple, and simple multicurves. In some sense, establishing this relation is the main goal of the paper. 

In the case that $\Sigma$ has no punctures this relation is pretty straight forward. Namely, we will prove that ``generic" multicurves in $\CS_{\gamma_0}$ have intersections with extremely small angles: 

\begin{sat}\label{closed-small-angles}
Let $\Sigma$ be a closed hyperbolic surface of genus $g\geq 2$ and let $\measuredangle(\gamma)\in(0,\frac\pi2]$ denote the largest angle among the self-intersections of a multicurve $\gamma\subset\Sigma$. Then 
\[
\lim_{L\to\infty}\frac{1}{L^{6g-6}}\left\vert\left\{
\begin{array}{c}
\gamma\subset \Sigma \text{ multicurve},  \iota(\gamma,\gamma)= k, \\ \measuredangle(\gamma)\geq\delta,\, \ell_\Sigma(\gamma)\leq L
\end{array}
\right\}\right\vert=0
\]
for every $k$ and every $\delta>0$.
\end{sat}

Hence, if $\Sigma$ is closed it follows from Theorem \ref{closed-small-angles} that generically elements of $\CS_{\gamma_0}$ have, if their length is large, extremely flat self-intersections. This means thus that there is a well-determined way to resolve the self-intersections of such  generic $\gamma$ to produce a simple multicurve which locally is almost parallel to the original curve. Doing so we obtain a map
$$\pi_{\epsilon,\gamma_0}:\CS_{\gamma_0}^\epsilon\to\CM\CL_\BZ(\Sigma)$$
from a {\em generic} subset $\CS_{\gamma_0}^\epsilon$ of $\CS_{\gamma_0}$ to the set $\CM\CL_\BZ(\Sigma)$ of simple multicurves in $\Sigma$. Here, genericity of $\CS_{\gamma_0}^\epsilon$ just means that 
$$\lim_{L\to\infty}\frac {\vert\{\gamma\in\CS_{\gamma_0}\setminus\CS_{\gamma_0}^\epsilon\vert\ell_\Sigma(\gamma)\le L\}\vert}{L^{6g-6+2r}}=0.$$
The map $\pi_{\epsilon,\gamma_0}$ maps multicurves in $\CS_{\gamma_0}^\epsilon$ to simple multicurves of basically the same length. By itself, this property already implies that $\pi_{\epsilon,\gamma_0}$ is finite-to-one. What is more remarkable is that the cardinality of the preimages $\vert\pi_{\epsilon,\gamma_0}^{-1}(\gamma)\vert$ is uniformly bounded from above (see section \ref{sec-themap} for precise statements). 

The basic idea of the proof of Proposition \ref{prop-sublimit} is to push the measures $\nu_{\gamma_0}^L$ via the map $\pi_{\epsilon,\gamma_0}$ to the space of measured laminations and compare them to the so obtained measures $\mu_{\epsilon,\gamma_0}^L$. The latter measures can be more concretely written as 
$$\mu_{\epsilon,\gamma_0}^L=\frac 1{L^{6g-6+2r}}\sum_{\gamma\in\CM\CL_\BZ}\vert\pi_{\epsilon,\gamma_0}^{-1}(\gamma)\vert\delta_{\frac 1L\gamma}$$
where $\delta_x$ is again the Dirac measure centred at $x$. We will show that the measures $\mu_{\epsilon,\gamma_0}^L$ and $\nu_{\gamma_0}^L$ get closer and closer to each other as $L$ increases (see Lemma \ref{alllimitsthesame}). In particular, to establish the convergence of $\nu_{\gamma_0}^L$ it suffices to prove that $\mu_{\epsilon,\gamma_0}^L$ converges. 

The key observation is that, since $\vert\pi_{\epsilon,\gamma_0}^{-1}(\gamma)\vert$ is uniformly bounded, any accumulation point $\mu$ of $(\mu_{\epsilon,\gamma_0}^L)$ is absolutely continuous with respect to the Thurston measure $\mu_{\Thu}$ on $\CM\CL(\Sigma)$. Since $\mu$ is also a limit of the mapping class group invariant measures $\nu_{\gamma_0}^L$, it is itself also mapping class group invariant. Thus we get from a theorem due to Masur \cite{Masur} asserting that $\mu_{\Thu}$ is ergodic, that the limit $\mu$ is actually a multiple of the Thurston measure, as we claimed in Proposition \ref{prop-sublimit}.

Note at this point that if the map $\pi_{\epsilon,\gamma_0}$ were equivariant under the mapping class group, then the existence of the limit $\lim_{L\to\infty}\mu_{\epsilon,\gamma_0}^L$ would directly follow from the work of Mirzakhani \cite{Maryam} on the distribution of simple multicurves. However, the map $\pi_{\epsilon,\gamma_0}$ is unfortunately not equivariant because its domain $\CS_{\gamma_0}^\epsilon$ is not mapping class group invariant. 

To partially by-pass this problem we will observe that if $\tau$ is an almost geodesic maximal train-track in $\Sigma$ then
\begin{equation}\label{eq-adele2}
\vert\pi_{\epsilon,\gamma_0}^{-1}(\phi(\gamma))\vert\ge\vert\pi_{\epsilon,\gamma_0}^{-1}(\gamma)\vert
\end{equation}
for all $\gamma$ carried by and filling $\tau$ and for every mapping class $\phi$ in the semi-group
$$\Gamma_\tau=\{\phi\in\Map(\Sigma)\vert\phi(\tau)\prec\tau\}$$
consisting of those mapping classes which map $\tau$ to a train-track carried by $\tau$. This is relevant because from \eqref{eq-adele2} and Proposition \ref{prop-sublimit} we get:

\begin{named}{Proposition \ref{prop-semigroup-limit}}
Let $\tau$ be a maximal recurrent train-track and $U\subset\{\lambda\in\CM\CL(\Sigma)\vert\lambda\prec\tau\}$ open with $\mu_{\Thu}(U)>0$ and $\mu_{\Thu}(\D U)=0$. Suppose also that the following holds:
\begin{quote}
(*) If $\CI\subset\{\gamma\in\CM\CL_\BZ(\Sigma)\vert\gamma\prec\tau\}$ is a non-empty $\Gamma_\tau$-invariant set of simple multicurves carried by $\tau$ then there is $\alpha>0$ with
$$\lim_{L\to\infty}\frac 1{L^{6g-6+2r}}\vert\CI\cap L\cdot U\vert=\alpha\cdot\mu_{\Thu}(U).$$
\end{quote}
Then the limit $\lim_{L\to\infty}\nu_{\gamma_0}^L$ exists.
\end{named}

The virtue of Proposition \ref{prop-semigroup-limit} is that it reduces the problem of showing that the measures $\nu_{\gamma_0}^L$ converge to a problem about distribution of simple multicurves. On the other hand, working with semigroups is harder than working with groups.

However, in the case of a punctured torus we can identify the semigroup $\Gamma_\tau$ with the positive semigroup $\SL_2\BN$ of the mapping class group $\Map(\Sigma)\simeq\SL_2\BZ$, the set of multicurves carried by $\tau$ with $\BN^2$, the space of measured laminations with $\BR^2/\pm 1$, and the Thurston measure with Lebesgue measure. In other words we are in a very concrete situation and we can use results of Maucourant \cite{pakito} to prove:

\begin{sat}\label{thm-density}
Every $\SL_2\BN$-invariant set $\CI\subset\BN^2$ has a density, meaning that there is $\alpha\in\BR$ with  
$$\lim_{L\to\infty}\frac 1{L^2}\vert \CI\cap L\cdot U\vert=\alpha\cdot\vol(U)$$
for any $U\subset\BR^2$ open and bounded by a rectifiable Jordan curve. Here $L\cdot U=\{v\in\BR^2\vert \frac 1Lv\in U\}$ is the set obtained by scaling $U$ by $L$ and $\vol(U)$ is the area of $U$ with respect to Lebesgue measure. 
\end{sat}

Theorem \ref{sat2} will follow from Theorem \ref{thm-density}, Proposition \ref{prop-semigroup-limit} and Proposition \ref{prop-reduction}.
\medskip

Before concluding the introduction, note that Theorem \ref{closed-small-angles} was only stated for closed surfaces. Indeed, the statement is wrong in general. In the presence of cusps there are small and large angles due to the fact that the intersections might well occur in the vicinity of the cusps. This means that the construction of the map $\pi_{\epsilon,\gamma_0}$ is going to be more subtle in the presence of cusps than in the closed case. In fact, section \ref{sec-radallas1} and section \ref{sec-radallas2} are basically devoted to constructing the map $\pi_{\epsilon,\gamma_0}$ and to study some of its properties. More concretely, in section \ref{sec-radallas1} we introduce a generalization of the notion of train-track, to which we give the beautiful Finnish name of {\em radalla}. A radalla is basically an immersed version of a train-track, and the main result of this section is Proposition \ref{prop-cars} which asserts that {\em for every $k$ and every radalla $\hat\tau$ there is a uniform upper bound for the cardinality of the set of multicurves carried by $\hat\tau$, with $k$ self-intersections and which represent a given solution of the switch equations}. In section \ref{sec-radallas2} we study the Hausdorff limits of sequences of multicurves with self-intersections and prove that for every $k$ there are finitely many almost geodesic radallas which carry all but finitely many multicurves with at most $k$ self-intersection. We use this fact to prove that the number $\vert\{\gamma\in\CS_{\gamma_0}\vert\ell_X(\gamma)\le L\}\vert$ grows polynomially of degree $L^{6g-6+2r}$ (Corollary \ref{poly-growth}) and to prove Theorem \ref{closed-small-angles}. Finally, in section \ref{sec-radallas2} we also construct the map $\pi_{\epsilon,\gamma_0}$. In section \ref{sec-measures} we study the measures $\nu_{\gamma_0}^L$ and $\mu_{\epsilon,\gamma_0}^L$ and prove Proposition \ref{prop-sublimit}, Proposition \ref{prop-reduction}, Corollary \ref{kor-ratio} and Proposition \ref{prop-semigroup-limit}. In section \ref{sec-torus} we discuss the case when $\Sigma$ is a once-punctured torus, prove Theorem \ref{thm-density} and conclude the proof of Theorem \ref{sat2}. 

\begin{bem}
Before concluding we would like to point out that all the results in this paper hold true if we replace the mapping class group $\Map(\Sigma)$ by one of its finite index subgroups $\Gamma$, that is, if we consider the set $\Gamma\cdot\gamma_0$ instead of the set $\CS_{\gamma_0}=\Map(\Sigma)\cdot\gamma_0$. The point is that we rely on Masur's theorem on the ergodicty of the Thurston measure and this result also holds true for finite index subgroups of the mapping class group.
\end{bem}

\medskip

\noindent{\bf Acknowledgements.} Many thanks are due to Maryam Mirzakhani for interesting comments and for sharing a draft of her paper \cite{Maryam-new}. We also thank Fran\c cois Maucourant for basically proving Proposition \ref{density-orbits} for us. We also thank Benjamin Bandli for interesting conversations. We are also grateful to the organizers of the Sixth Iberoamerican Congress on Geometry, especially to Ara Basmajian, because it was during this event that we began this work. In fact, we were motivated by a beautiful talk by Moira Chas.


\section{}\label{sec-radallas1}

Train-tracks are a key tool to study the structure of the set of all simple curves on surfaces. In this section we introduce a small variation, basically an immersed version of a train-track, which we will use to study curves with self-intersections. We refer to \cite{Penner-Harer,Papadopoulos,Hatcher} for basic facts about train-tracks.

\subsection{Train-tracks and radallas}
By a smoothly embedded 1-complex $\tau$ in a surface $\Sigma$ we mean a finite embedded complex whose edges are smoothly embedded arcs with well-defined tangent lines at the end-points. We moreover require that for every vertex $v$ all the lines tangent to edges adjacent to it agree -- we denote this line by $T_v\tau$. 

If $\tau$ is such an embedded 1-complex and $v\in\tau$ is a vertex, then the set of (germs of) edges adjacent to $v$ can be divided into two sets, according to the two directions of $T_v\tau$. We say that $\tau$ is a pre-train-track if these two sets are non-empty for each vertex -- equivalently, every vertex $v$ of $\tau$ is contained in the interior of a smoothly embedded arc $I\subset\tau$. A {\em complementary region} $\Delta$ of a pre-train-track $\tau$ is the metric completion of a connected component of $\Sigma\setminus\tau$. Note that the boundary of $\Delta$ is smooth except at a finite number of cusp points corresponding to vertices of $\tau$. If $v$ happens to be a trivalent vertex, then $v$ corresponds to a single cusp in a single complementary region. A pre-train-track $\tau$ is a {\em train-track} if no complementary region is a disk with at most 2 cusps or a disk with one puncture and no cusp.

\begin{bem}
We will only consider trivalent train-tracks, meaning that all vertices have degree $3$.
\end{bem}

By definition, train-tracks are embedded in $\Sigma$ -- we now define an immersed version:

\begin{defi*}
A {\em radalla}\footnote{According to Google Translate, {\em radalla} means {\em track} in Finnish. According to our Finnish connection it is mostly used in the form {\em olla radalla} as in {\em to be on track to pick up girls in bars} but we ignored it when we coined this term.} in a surface $\Sigma$ is a triple $(\hat\tau,\tau,\phi:\hat\tau\looparrowright\Sigma)$ where
\begin{enumerate}
\item $\hat\tau$ is a finite graph and $\tau\subset\hat\tau$ is a subgraph containing all vertices. 
\item $\phi:\hat\tau\looparrowright\Sigma$ is a smooth immersion whose restriction to $\tau$ is an embedding and $\phi(\tau)$ is a train-track in $\Sigma$.
\item If $e$ is an edge of $\hat\tau\setminus\tau$, then 
\begin{itemize}
\item $\phi(e)$ is contained in a complementary region $\Delta$ of the train-track $\phi(\tau)$.
\item The end points of $e$ are mapped into cusps of $\Delta$.
\item $\phi\vert_e$ cannot be homotoped relative to the endpoints to a map with image in $\phi(\tau)$.
\item $\phi\vert_e$ cannot be homotoped to $\phi\vert_{e'}$ relative to the endpoints for any edge $e'$ of $\hat\tau\setminus\tau$ distinct from $e$.
\end{itemize}
\end{enumerate}
\end{defi*}

We will often denote the radalla $(\hat\tau,\tau,\phi:\hat\tau\looparrowright\Sigma)$ just by $\hat\tau$ and we will identify the subgraph $\tau$ with the train-track $\phi(\tau)$. In fact, one should think of a radalla $\hat\tau$ as a diagonal extension of the underlying train-track $\tau$ where the added diagonals are possibly neither simple nor disjoint -- compare with figure \ref{figradalla}. 

\begin{figure}[h]
\includegraphics[width=8cm, height=4cm]{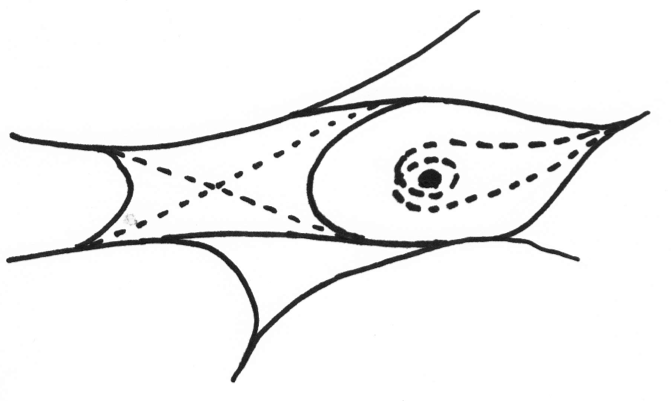}
\caption{Part of the image of a radalla $\hat\tau$ -- the solid lines are part of the underlying train-track $\tau$ and the black dot is a cusp.}
\label{figradalla}
\end{figure}

Recall that an immersed curve $\gamma:\BS^1\to\Sigma$ is isotopic to a second curve $\gamma':\BS^1\to\Sigma$ if there is a smooth map 
$$\BS^1\times[0,1]\to\Sigma,\ \ (t,s)\mapsto\gamma_s(t)$$
such that $\gamma_0=\gamma$, $\gamma_1=\gamma'$ and $\gamma_s$ is an immersion for all $s$. Note that by definition any curve isotopic to an immersed curve is immersed as well. 

\begin{defi*}
Let $\hat\tau=(\hat\tau,\tau,\phi:\hat\tau\looparrowright\Sigma)$ be a radalla. An immersed but possibly non-simple curve $\gamma:\BS^1\to\Sigma$ is {\em carried} by $\hat\tau$ if there is a map $\gamma':\BS^1\to\hat\tau$ such that $\gamma$ and $\phi\circ\gamma'$ are isotopic.
\end{defi*}
 
Curves carried by a radalla behave very much like curves carried by a train-track. For instance, if $\Sigma$ is endowed with a hyperbolic metric then there is a constant $L=L(\hat\tau,\Sigma)$ such that for every curve $\gamma:\BR\to\hat\tau$ for which $\phi\circ\gamma:\BR\to\Sigma$ has constant velocity, any lift $\widetilde{\phi\circ\gamma}:\BR\to\tilde\Sigma=\BH^2$ is an $L$-bilipschitz embedding. In particular, if $\gamma:\BS^1\to\hat\tau$ is such that $\phi\circ\gamma$ is a smooth immersion, then $\gamma$ is an essential curve in $\Sigma$. Moreover, if $\gamma,\gamma':\BS^1\to\hat\tau$ where $\phi\circ\gamma$ and $\phi\circ\gamma'$ are homotopic smooth immersions, then there is $h:\BS^1\to\BS^1$ with $\gamma'=\gamma\circ h$.

Recall that a train-track $\tau$ is carried by another train-track $\tau'$ -- that is $\tau\prec\tau'$ -- if the embedding of $\tau$ into $\Sigma$ is smoothly isotopic to an immersion $\tau\to\Sigma$ with image contained in $\tau'$. This definition carries over to radallas as follows: we say that {\em $\hat\tau=(\hat\tau,\tau,\phi:\hat\tau\looparrowright\Sigma)$ is carried by $\hat\tau'=(\hat\tau',\tau',\phi':\hat\tau'\looparrowright\Sigma)$} -- and write $\hat\tau\prec\hat\tau'$ -- if there is a smooth immersion $\psi:\hat\tau\to\hat\tau'$ mapping $\tau$ into $\tau'$ and such that $\phi'\circ\psi$ is homotopic to $\phi$. Clearly, being carried is a transitive property, meaning for instance that if multicurve $\gamma\prec\hat\tau$ and if $\hat\tau\prec\hat\tau'$ then $\gamma\prec\hat\tau'$.
\medskip

While train-tracks and radallas are topological objects, we will be mostly interested in those that are geometrically well-behaved. A radalla $\hat\tau=(\hat\tau,\tau,\phi:\hat\tau\looparrowright\Sigma)$ is {\em $\epsilon$-geodesic} if the image $\phi(e)$ of each edge $e$ of $\hat\tau$ has at most geodesic curvature $\epsilon$ and length at least $\frac 1\epsilon$. Note that it follows from this definition that, as long as $\epsilon$ is small enough, every curve carried by an $\epsilon$-geodesic radalla has at least length $\frac 1{2\epsilon}$. Also note that, since all the vertices of $\hat\tau$ are also vertices of $\tau$, if the train-track $\phi(\tau)$ is $\epsilon$-geodesic, then the radalla is isotopic to an $\epsilon$-geodesic radalla.

\subsection{Thickenings of radallas}

Train-tracks can be consider as graphs or as band complexes, and the same is true for radallas. We recall briefly the definition of a band complex. For us, a {\em band} is just a rectangle $[0,a]\times[0,b]$ with long horizontal sides and short vertical sides ($a>>>b$), foliated by the vertical segments $\{t\}\times[0,b]$. A {\em band complex} is a space $X$ obtained by gluing finitely many bands together along the short vertical sides in such a way that in the end the whole boundary of $X$ consists of the horizontal components of the bands. Note that the vertical foliations of the bands match up to a foliation $\CF_X$ of the band complex $X$. 

\begin{defi*}
A {\em thickening} of a radalla $(\hat\tau,\tau,\phi:\hat\tau\looparrowright\Sigma)$ is a triple $(X,\iota,\bar\phi)$ where
\begin{itemize}
\item $X$ is a surface endowed with a structure of a band complex,
\item $\iota:\hat\tau\to X$ is an embedding transversal to the vertical foliation $\CF_X$ such that each band contains a leaf of $\CF_X$ which meets $\iota(\hat\tau)$ exactly once, and
\item $\bar\phi:X\to\Sigma$ is an immersion with $\phi=\bar\phi\circ\iota$. 
\end{itemize}
\end{defi*}

We note that each radalla has a thickening and that thickenings are unique up to homeomorphism of $X$ and isotopy of $\bar\phi$. Note also that a curve $\gamma:\BS^1\to\Sigma$ is carried by a radalla $(\hat\tau,\tau,\phi:\hat\tau\looparrowright\Sigma)$ if and only if whenever $(X,\iota,\bar\phi)$ is a thickening there is an immersion $\bar\gamma:\BS^1\to X$ transversal to the vertical foliation $\CF_X$ and with $\bar\phi\circ\bar\gamma$ isotopic to $\gamma$. The curve $\bar\gamma:\BS^1\to X$ is unique up to isotopy. 
\medskip

Recall that by a {\em multicurve} $\lambda$ in $\Sigma$ we mean a closed immersed 1-manifold, each one of whose components represents a primitive non-peripheral element in $\pi_1(\Sigma)$, and that $\iota(\lambda,\lambda)$ is the minimum number of self-intersections of those $\lambda'$ in general position and isotopic to $\lambda$. The condition that each component of a multicurve represents a non-trivial {\em primitive} element in $\pi_1(\Sigma)$ amounts to asserting that each multicurve is determined by the associated geodesic current. Moreover, the self-intersection number agrees with the self-intersection number when we see multicurves as currents.
\medskip

Suppose that the multicurve $\lambda$ in $\Sigma$ is carried by a radalla $\hat\tau=(\hat\tau,\tau,\phi:\hat\tau\looparrowright\Sigma)$, meaning that each component is carried by $\hat\tau$. Let $X=(X,\iota,\bar\phi)$ be a thickening of $\hat\tau$ and $\bar\lambda\subset X$ a multicurve transversal to the vertical foliation $\CF_X$ such that $\bar\phi(\bar\lambda)$ is isotopic to $\lambda$. Since $\bar\lambda$ is unique up to isotopy we get that the self-intersection number
\begin{equation}\label{eq-poor-greeks}
\iota_X(\lambda,\lambda)\stackrel{\hbox{\tiny def}}=\iota(\bar\lambda,\bar\lambda)
\end{equation}
of $\bar\lambda$ with itself depends only on $\lambda$ and $X$. We note that 
$$\iota(\lambda,\lambda)\ge\iota_X(\lambda,\lambda)$$
but that in general equality does not hold: for instance the self-intersections arriving from the edges of $\hat\tau$ which are not embedded under $\phi$ are not counted when computing $\iota_X(\lambda,\lambda)$.

\subsection{Switch equations}

Let $(\hat\tau,\tau,\phi:\hat\tau\looparrowright\Sigma)$ be a radalla and denote by $E(\hat\tau)$ and $V(\hat\tau)$ the sets of edges and vertices of $\hat\tau$. At each vertex $v$ choose an orientation of $T_v\hat\tau$ and denote by $E_v^+(\hat\tau)$ (resp. $E_v^-(\hat\tau)$) all the (germs of) edges starting at $v$ in positive (resp. negative) direction. These choices yield a linear map
$$W_{\hat\tau}:\BR^{E(\hat\tau)}\to\BR^{V(\hat\tau)},\ \ (a_e)_{e\in E(\hat\tau)}\mapsto\left(\sum_{e\in E_v^+(\hat\tau)}a_e-\sum_{e\in E_v^-(\hat\tau)}a_e\right)_{v\in V(\hat\tau)}$$
We refer to the entries of the elements in $\BR_+^{E(\hat\tau)}$ as {\em weights}. Elements in $\ker(W_{\hat\tau})$ are said to satisfy the {\em switch equations}. More concretely, $w\in\BR_+^{E(\hat\tau)}$ satisfies the switch equations if at every vertex the sum of positive weights equal the sum of negative weights.

Every curve $\gamma:\BS^1\to\hat\tau$ with $\phi\circ\gamma$ an immersion yields a solution $\omega_\gamma$ to the weight equation by associating to each edge $e\in E(\hat\tau)$ the cardinality of $\gamma^{-1}(x)$ for some interior point $x\in e$. Note that $\omega_\gamma=\omega_{\gamma'}$ if $\phi\circ\gamma$ and $\phi\circ\gamma'$ are homotopic. The vector $\omega_\gamma$ associated to a multicurve is the sum of the vectors associated to the individual components.

Basically, the difference between a train track and a radalla is that the later is not embedded, meaning that there are edges which cross themselves or which cross another edge. This implies that often curves carried by a radalla are going to have self-intersections. In fact, if $e$ and $e'$ are (possibly equal) edges of a radalla $\hat\tau$ with $\iota(e,e')=r$ and if $\gamma:\BS^1\to\hat\tau$ is an immersion with associated vector of weights $\omega_\gamma$ then $\gamma$ has at least $r\cdot\omega_\gamma(e)\cdot\omega_\gamma(e')$ self-intersections. Here $\iota(e,e')$ is the minimal number of interior intersection points of edges isotopic to $e$ and $e'$ in the complement of the underlying train-track and relative to the respective endpoints. 

\begin{bem}
In fact, it is possible to establish a more precise version of this last fact. Namely, if $(\hat\tau,\tau,\phi:\hat\tau\looparrowright\Sigma)$ is a radalla with thickening $(X,\iota,\bar\phi)$, then we have
$$\iota(\lambda,\lambda')=\iota_X(\lambda,\lambda')+\sum_{\{e,e'\}\subset E(\hat\tau)\setminus E(\tau)}\iota(e,e')\cdot\omega_\lambda(e)\cdot\omega_{\lambda'}(e')$$
for  any two, possibly equal, multicurves $\lambda,\lambda'$ carried by $\hat\tau$. Here $\iota_X(\lambda,\lambda')$ is defined as in \eqref{eq-poor-greeks}.
\end{bem}

\subsection{The meat}

It is well-known that a simple multicurve $\gamma$ carried by a train-track is determined by the associated vector of weights $\omega_\gamma$. It is on the other hand easy to find examples showing that this is no longer true if the multicurves under consideration are not simple. We will see however that as long as one only allows a bounded number of self-intersections, then $\gamma$ is determined by $\omega_\gamma$ up to bounded indeterminacy. In fact, Proposition \ref{prop-cars} is the most important result of this section: 

\begin{prop}\label{prop-cars}
For every radalla $(\hat\tau,\tau,\phi:\hat\tau\looparrowright\Sigma)$ and every $k$ there is $K$ such that for every integral positive solution $\omega\in\BZ_+^{E(\hat\tau)}\subset\BR^{E(\hat\tau)}$ of the switch equation there are at most $K$ homotopy classes of multicurves $\gamma$ carried by $\hat\tau$ with $\omega_\gamma=\omega$ and with $\iota(\gamma,\gamma)= k$.
\end{prop}

The proof of Proposition \ref{prop-cars} is pretty long -- we suggest that the reader skips over it in a first reading.

\begin{proof}
The basic strategy of the proof is to show that each multicurve $\gamma$ carried by $\hat\tau=(\hat\tau,\tau,\phi:\hat\tau\looparrowright\Sigma)$, with $\iota(\gamma,\gamma)= k$ and with $\omega_\gamma=\omega$, can be isotoped into ``normal form", and that the number of curves in the said normal form is bounded in terms of the radalla $\hat\tau$ and the bound $k$ on the number of self-intersections.

To begin with we choose a thickening $X=(X,\iota,\bar\phi)$ of $\hat\tau$. We will draw a pattern, of which we think of as millimeter paper, on a subset $\bar X$ of $X$. Before describing the construction in words, we refer the reader to figure \ref{fig2} and suggest to keep this drawing in mind while reading the following lines. Well, starting with the construction denote by $B(e)$ the band of $X$ corresponding to the edge $e$ of $\hat\tau$. Draw $\omega_e$ disjoint segments on $B(e)$ that are transversal to the vertical foliation of $B(e)$, joins the two vertical boundary components, and such that when we consider the union of all these segments over all bands we obtain a simple multicurve $\Omega$ in $X$ with $\omega_\Omega=\omega$. Now, we choose for each edge $e$ of $\hat\tau$ a finite set $\CL(e)$ consisting of at least $200k$ vertical leaves contained in the band $B(e)$ and we set $\CL=\cup\CL(e)$ -- abusing notation we will use $\CL$ (resp. $\CL(e)$) to refer both to the finite set of leaves and to the union of those leaves as a subset of $X$. The multicurve $\Omega$ and the set $\CL$ determine a tiling by closed squares of a subset $\bar X$ of $X$, namely the closure of the union of those squares in $X$ whose boundary consists of two subsegments of $\CL$ and two subsegments of $\Omega$. The boundary $\D\bar X$ consists of finitely many vertical and finitely many horizontal subsegments -- we refer to the vertical subsegments as the {\em exceptional vertical segments}. Note that the exceptional segments appear only at the vertices of $\hat\tau$ -- more concretely, if a vertex is such that $s$ edges of $\hat\tau$ merge into an edge, then we find there $s-1$ exceptional vertical segments.

\begin{figure}[h]
\includegraphics[width=8cm, height=4cm]{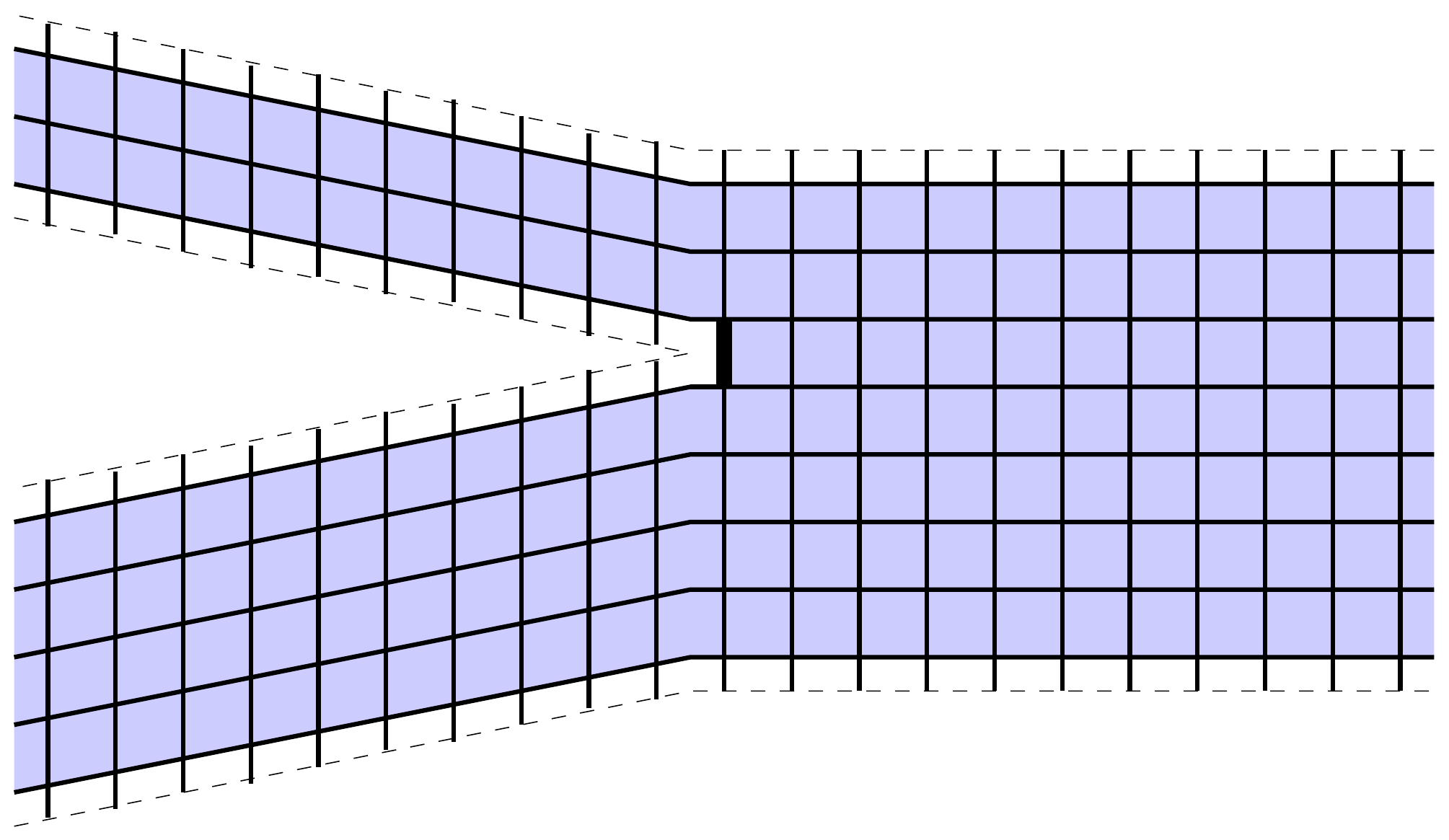}
\caption{The milllimeter paper pattern on $X$: the dotted lines are the boundary of $X$, the solid vertical segments are the leaves in $\CL$, the solid horizontal segments are the trace of the multicurve $\Omega$, the shaded region is $\bar X$, and finally the thick and short vertical segment is the exceptional segment corresponding to the vertex in the picture.}
\label{fig2}
\end{figure}

We will refer to the closed squares in $\bar X$ as {\em tiles}. For lack of a better word, if $Q$ is a tile we will refer to the two horizontal sides and to the two diagonals of $Q$ as being non-vertical. A multicurve $\gamma:\BS^1\sqcup\dots\sqcup\BS^1\to X$ is in {\em preliminary normal form} (cf. figure \ref{fig3}) if
\begin{itemize}
\item the set of self-intersections is discrete, 
\item the image is the union of non-vertical segments, and 
\item $\gamma$ is $C^0$-close to a smooth multicurve in $X$ which is transversal to the vertical foliation $\CF_X$.
\end{itemize}
Observe that every curve in preliminary normal form is carried by $\hat\tau$ and that by construction we have $\omega_\gamma(e)\le\omega(e)$ for every edge $e$ and every such multicurve $\gamma$. Note also that since we are only going to be interested in homotopy classes of multicurves, we will identify multicurves in preliminary normal form with the same image.
\medskip

\begin{figure}[h]
\includegraphics[width=8cm, height=4cm]{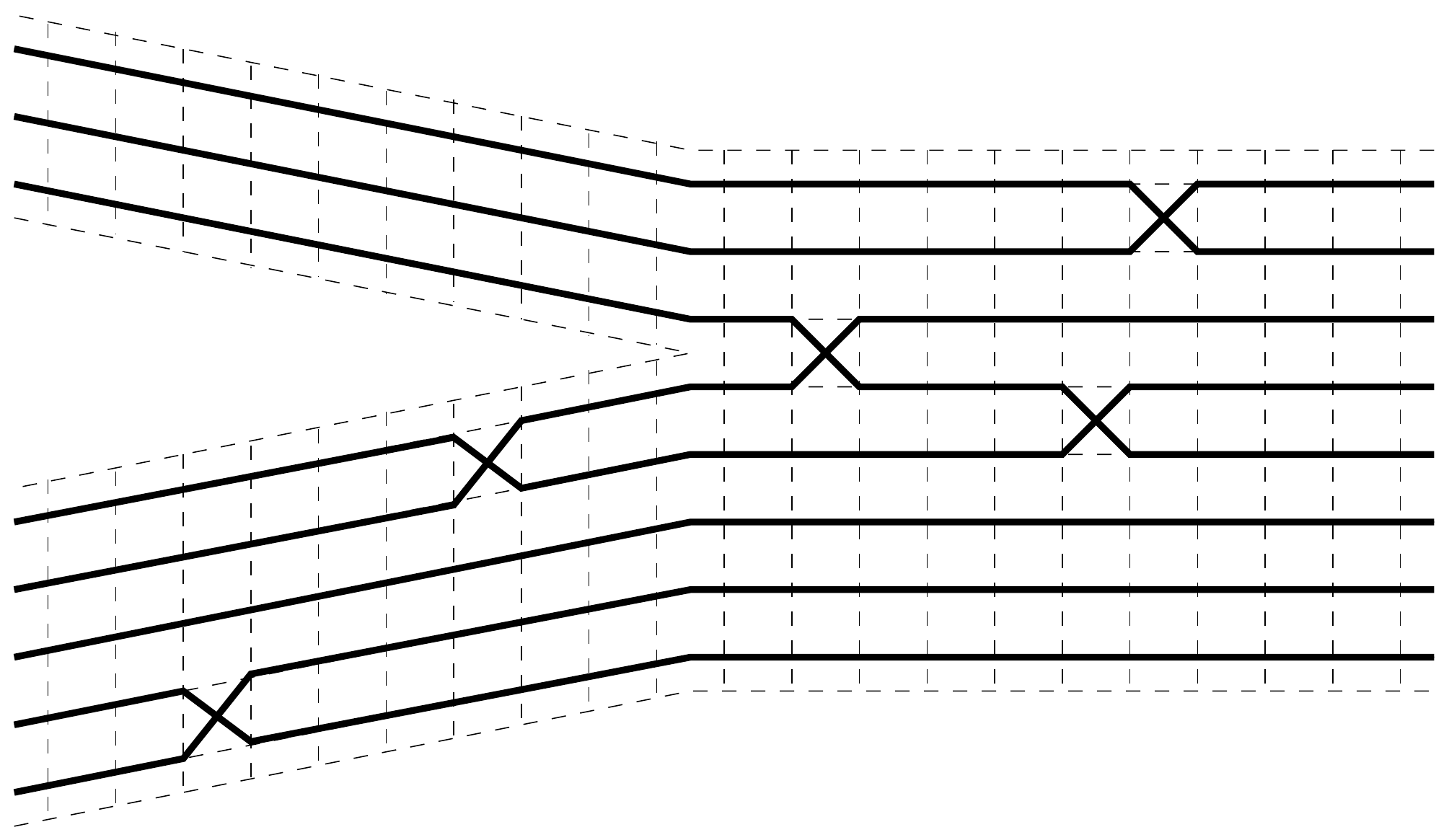}
\caption{A curve $\gamma$ in preliminary normal form with weight vector $\omega_\gamma=\omega$.}
\label{fig3}
\end{figure}

\noindent{\bf Claim 1.} Every multicurve $\gamma$ in $X$ transversal to the vertical foliation $\CF_X$, with $\omega_\gamma=\omega$, and with $\iota_X(\gamma,\gamma)\le k$ is isotopic (transversally to $\CF_X$) to a multicurve $\gamma'$ in preliminary normal form and with $\iota_X(\gamma',\gamma')=\iota_X(\gamma,\gamma)$.

\begin{proof}[Proof of Claim 1.]
We can perturb $\gamma$ so that, while keeping all its listed properties, we also have that $\gamma$ is contained in the tiled part $\bar X$ of $X$, no crossing of $\gamma$ lies on $\CL$, and there is at most a single crossing between any two consecutive leaves of $\CL$. Here we say that two leaves $\ell,\ell'$ of $\CL$ are consecutive if there is a square whose vertical boundary is contained in $\ell\cup\ell'$.

Note that all of this means that each $\ell\in\CL(e)$ meets $\gamma$ in $\omega_\gamma(e)=\omega(e)$ points. We can thus isotope $\gamma$ so that, while keeping all its properties so far, it meets (for every $e$) each leave $\ell\in\CL(e)$ in the points $\ell\cap\Omega$, i.e. in the vertices of the tiling of $\bar X$. Now, the image of $\gamma$ consists of segments $I$ which join points in $\CL\cap\Omega$ and whose interiors are disjoint from $\CL$. Moreover, the two endpoints of any such segment $I$ are contained in some tile. Replacing each segment $I$ by the corresponding straight segment in the tile we obtain a curve in preliminary normal form, as we wanted to construct.
\end{proof}

Having brought multicurves in preliminary normal form is not enough to prove Proposition \ref{prop-cars} because the number of multicurves in such form depends not only on the radalla and on the number $k$ of allowed crossings, but also on the entries of the vector $\omega$. To avoid this dependence, we are going to associate a complexity $\kappa(\gamma)$ to every multicurve $\gamma$ in preliminary normal form and with $\iota_X(\gamma,\gamma)\le k$. Note that each such curve $\gamma$ determines a collection $\CC_\gamma$ of $\iota_X(\gamma,\gamma)$ closed tiles, namely the tiles containing a crossing -- let $\vert\CC_\gamma\vert$ be the support of $\CC_\gamma$, i.e. the union of the closed tiles therein.

We set the complexity of $\gamma$ to be
\begin{align*}
\kappa(\gamma)=&\hbox{the number of tiles in $\CC_\gamma$ contained in those connected components}\\
&\hbox{of $\vert\CC_\gamma\vert$ which do NOT contain an exceptional vertical segment.}
\end{align*}
We will say that a multicurve $\gamma$ is in {\em normal form} if it is in preliminary normal form and has vanishing complexity $\kappa(\gamma)=0$. We are going to prove:
\smallskip

\noindent{\bf Claim 2.} Every multicurve $\gamma$ in $X$ transversal to the vertical foliation $\CF_X$, with $\omega_\gamma=\omega$, and with $\iota_X(\gamma,\gamma)\le k$ is isotopic (transversally to $\CF_X$) to a multicurve $\gamma'$ in normal form, and with $\iota_X(\gamma',\gamma')=\iota_X(\gamma,\gamma)$.
\smallskip

Assuming Claim 2 for a moment, we conclude the proof of Proposition \ref{prop-cars}. Any multicurve $\lambda$ carried by $\hat\tau$ lifts to a multicurve $\bar\lambda\subset X$, meaning that $\bar\phi(\bar\lambda)$ isotopic to $\lambda$. We know moreover that 
$$\iota_X(\bar\lambda,\bar\lambda)\le\iota(\lambda,\lambda)= k$$
Now, from Claim 2 we get that $\bar\lambda$ is isotopic to a multicurve $\gamma\subset X$ in normal form and with $\iota_X(\gamma,\gamma)\le k$. The claim of Proposition \ref{prop-cars} will follow once we prove that the number of choices for $\gamma$ is bounded just in terms of $k$ and the radalla $\hat\tau$. To see that this is the case endow the quadrangulated part $\bar X$ with a metric which makes each tile isometric to the euclidean square of diameter $1$. Now, if $\gamma$ is a curve in preliminary normal form, with $\iota_X(\gamma,\gamma)\le k$ we have that each connected component of $\vert\CC_\gamma\vert$ has at most diameter $k$. If $\gamma$ is in normal form then each component of $\vert\CC_\gamma\vert$ contains an exceptional vertical segment, meaning that all the crossings of $\gamma$ are located in one of the tiles within distance $k$ of one of these exceptional vertical segments. Since the number of exceptional vertical segments just depends on the radalla $\hat\tau$ we obtain that all the crossings of $\gamma$ are in a set of tiles whose cardinality just depends on $\hat\tau$ and $k$, as we needed to show.

It remains to prove Claim 2.

\begin{proof}[Proof of Claim 2]
By Claim 1 we know that $\gamma$ is isotopic to a multicurve $\gamma^{(0)}\subset X$ in preliminary normal form and with $\iota_X(\gamma^{(0)},\gamma^{(0)})=\iota_X(\gamma,\gamma)$. Among all choices for $\gamma^{(0)}$, consider those with minimal complexity $\kappa(\gamma^{(0)})$, and among those with minimal complexity, suppose that the number of connected components of $\vert\CC_{\gamma^{(0)}}\vert$ is also minimal. We will show that $\kappa(\gamma^{(0)})=0$, meaning that $\gamma^{(0)}$ is in normal form. Seeking a contradiction suppose that there is a component $\vert\CC^*_{\gamma^{(0)}}\vert$ of $\vert\CC_{\gamma^{(0)}}\vert$ which does not contain an exceptional vertical segment. Choose one of the tiles $T$ forming $\vert\CC^*_{\gamma^{(0)}}\vert$ and orient it transversally to the vertical foliation but otherwise arbitrarily -- for the sake of concreteness we will refer to the positively oriented side as ``left". Note that this orientation of $T$ induces an orientation of each other tile in the connected component $\vert\CC^*_{\gamma^{(0)}}\vert$. We can now isotope the curve $\gamma^{(0)}$ to a curve $\gamma^{(1)}$ by shifting the crossings in $\vert\CC^*_{\gamma^{(0)}}\vert$ to the left. More precisely $\gamma^{(1)}$ is the curve in preliminary normal form with the same number of crossings as $\gamma^{(0)}$, such that each crossing of $\gamma^{(0)}$ which does not lie in $\vert\CC^*_{\gamma^{(0)}}\vert$ is still a crossing of $\gamma^{(1)}$, and finally such that each crossing of $\gamma^{(0)}$ contained in $\vert\CC^*_{\gamma^{(0)}}\vert$ has been replaced by a crossing on the tile to the left -- compare with figure \ref{fig4}. 

\begin{figure}[h]
\includegraphics[width=10cm, height=2.5cm]{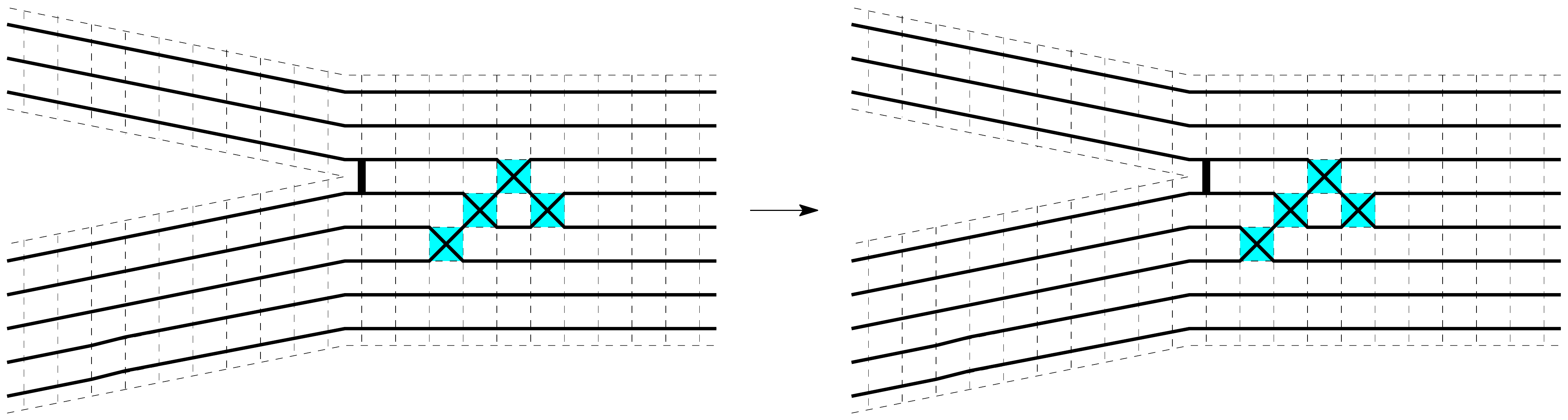}
\caption{The process by which $\gamma^{(1)}$ is obtained from $\gamma^{(0)}$.}
\label{fig4}
\end{figure}

The sets of tiles $\CC_{\gamma^{(1)}}$ and $\CC_{\gamma^{(0)}}$ containing the crossings of $\gamma^{(1)}$ and $\gamma^{(0)}$  are identical, besides the fact that the set of tiles $\CC^*_{\gamma^{(0)}}$ has been shifted to the left - denote the new set by $\CC^*_{\gamma^{(1)}}$. If the set $\vert\CC^*_{\gamma^{(1)}}\vert$ does not touch neither another component of $\vert\CC_{\gamma^{(1)}}\vert$, nor contains a exceptional vertical segment, then we can repeat this process and obtain curves $\gamma^{(2)}$ and so on. For instance, in the example presented in figure \ref{fig4} one can repeat this process 4 times but no more -- compare with figure \ref{fig5}.

\begin{figure}[h]
\includegraphics[width=12cm, height=2cm]{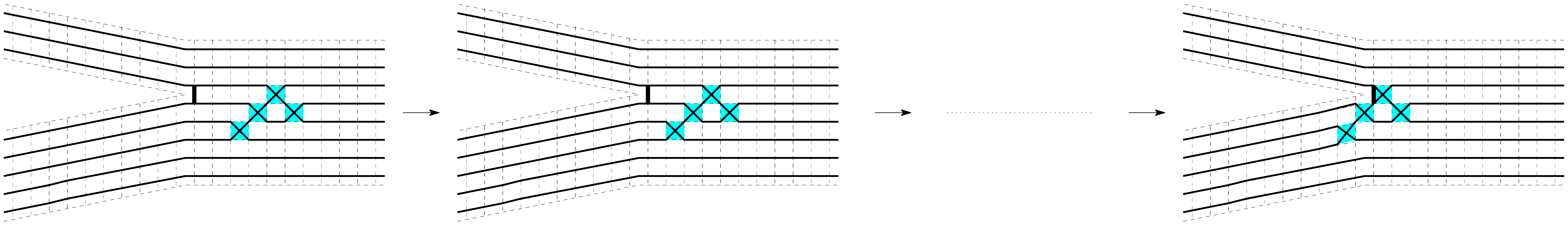}
\caption{The shaded tiles can be moved 4 times to the left but no more because the cloud corresponding to $\gamma^{(4)}$ contains an exceptional vertical segment.}
\label{fig5}
\end{figure}

Note that the minimality assumptions on our curve $\gamma^{(0)}$ imply that the process in question can be repeated infinitely often. This implies that there is some $s$ such that $\CC^*_{\gamma^{(s)}}=\CC^*_{\gamma^{(0)}}$ because there are only finitely many tiles and hence only finitely many configurations of tiles. But then, this implies that $\bar X$ contains a closed annulus
$$A=\cup_{i=0}^s\vert\CC^*_{\gamma^{(i)}}\vert$$
made out of tiles and such that the only (closed) tiles in $\vert\CC_{\gamma^{(0)}}\vert$ which intersect $A$ are those in $\CC^*_{\gamma^{(0)}}$. This implies that every component of $\gamma^{(0)}$ which meets $A$ is contained therein. Since we are assuming that $\gamma^{(0)}$ realizes $\iota_X(\gamma^{(0)},\gamma^{(0)})=\iota_X(\gamma,\gamma)$ and since there are crossings in $A$, it follows that some component of $\gamma^{(0)}\cap A$ represents a multiple of the soul of $A$, contradicting the assumption that $\gamma^{(0)}$ is isotopic to the multicurve $\gamma$, and hence that each of its components is  primitive in $\pi_1$. This proves Claim 2.
\end{proof}

Having proved Claim 2, we have also proved Proposition \ref{prop-cars}.
\end{proof}

Continuing with the same notation as in the proof of Proposition \ref{prop-cars}, note that the number $s(\hat\tau,\omega,k)$ of isotopy classes of multicurves $\gamma$ carried by the radalla $\hat\tau=(\hat\tau,\tau,\phi:\hat\tau\looparrowright\Sigma)$, with $\omega_\gamma=\omega$ and with $\iota(\gamma,\gamma)=k$ can be algorithmically computed. In fact, when $k$ is small the quantity $s(\hat\tau,\omega,k)$ does not depend on $\omega$ as long as $\omega(e)$ is bounded from below by some threshold to ensure that all possibilities can be realized.

Suppose for the sake of concreteness that the radalla $\hat\tau$ is actually a trivalent train-track (meaning that all vertices have degree $3$ and that $\hat\tau=\tau$) and that $\omega\in\BN_+^{E(\hat\tau)}$ is such that the weight of each edge is relatively large, say $\omega(e)\ge 10$ for all $e\in E(\hat\tau)$. Then, since simple multicurves in train-tracks are determined by the associated weights we have
$$s(\hat\tau,\omega,0)=1$$
Things are more complicated if we allow for intersections, and in fact we have
$$s(\hat\tau,\omega,1)=\frac V2$$
where $V$ is the number of vertices of the train-track $\tau=\hat\tau$. Indeed, if a multicurve $\gamma$ carried by $\hat\tau$ is in normal form and satisfies $\omega_\gamma=\omega$ and $\iota(\gamma,\gamma)=1$, then the unique crossing of $\gamma$ has to be in one of the $V$ tiles adjacent to an exceptional segment. Conversely, for each one of these tiles we get one such curve, meaning that there are exactly $V$ curves $\gamma\prec\hat\tau$ in normal form, with $\iota(\gamma,\gamma)=1$ and $\omega_\gamma=\omega$. On the other hand, $s(\hat\tau,\omega,1)=\frac V2$ because each multicurve $\gamma$ in normal form with $\omega_\gamma=\omega$ and $\iota(\gamma,\gamma)=1$ is isotopic to precisely another one such multicurve $\gamma'$ obtained as follows: let $T_0$ be the tile containing the crossing of $\gamma_0=\gamma$, let $T_1$ be the tile adjacent to $T_0$ and opposite to the exceptional segment contained in $T_0$, and let $\gamma_1$ be the multicurve in preliminary normal form, with $\omega_{\gamma_1}=\omega$ and with a single crossing in $T_1$. Then let $T_2$ be the tile adjacent to $T_1$ and opposite to $T_0$ and let $\gamma_2$ be the multicurve in preliminary normal form with $\omega_{\gamma_2}=\omega$ and with a single crossing in $T_2$. Define inductively $T_3,T_4,\dots$ pushing the tile away from the exceptional segment and let $\gamma_3,\gamma_4,\dots$ be the corresponding curves. This process has to end with $\gamma'=\gamma_k$ when $T_k$ is adjacent to a second exceptional segment. Note that $\gamma'$ is in normal form, that $\gamma'\neq\gamma$, and that $\gamma$ and $\gamma'$ are isotopic.

A similar computation can be done if one counts multicurves with $2$ self-intersections -- one obtains that
$$s(\hat\tau,\omega,2)=\frac{V(V+6)}8$$
Leaving the details of the computation of $s(\hat\tau,\omega,2)$ to the reader, we just sketch a possible approach. We think of the two crossings as railroad cars, assign to each one of them a weight and a direction and let them run around subject to the condition that if they touch each other then they get stuck together and travel in the direction of the heavier one. The argument in the proof of Claim 2 shows that at some point they both have to get stranded at an exceptional segment (cf. with figure \ref{fig5}). In this way we associate, after choosing weights and directions, to each multicurve in preliminary normal form a multicurve in normal form satisfying additional condition on the directions and weight of the cars. There are $V(V+6)$ such say "labeled normal forms" and 8 possible distributions of labels, which implies that $s(\hat\tau,\omega,2)=\frac{V(V+6)}8$, as we claimed.

\begin{bem}
The threshold $\omega(e)\ge 10$ for every edge is very generous -- in fact, for the two treated cases $k=1,2$, it would have sufficed to require that $\omega(e)\ge k+2$ for every edge $e$.
\end{bem}


\section{}\label{sec-radallas2}

With the same notation as all along let
\begin{equation}\label{eqdefset}
\CS_k=\{\gamma\text{ multicurve in }\Sigma\text{ with }\iota(\gamma,\gamma)=k\}
\end{equation}
be the set of all multicurves in $\Sigma$ with $k$ self-intersections. In this section we will show that there is a finite collection of radallas that carry all but finitely many elements in $\CS_k$. In the absence of cusps most element in $\CS_k$ are in fact carried by maximal train-tracks -- in the presence of cusps they are carried by radallas whose associated train-track is maximal and has finitely many extra edges around the punctures. Below we will see that these facts together with Proposition \ref{prop-cars} prove the polynomial growth of the number of elements in $\CS_k$ of length $\ell_\Sigma\le L$. We will also describe how one can associate a simple multicurve to each generic (in a precise sense) multicurve with self-intersections.

\subsection{Finding radallas}
The basic idea used to prove that there is a finite collection of radallas that carry all but finitely many elements in $\CS_k$ is to consider the possible Hausdorff limits of sequences of such multicurves and find radallas that carry their limits. 

We start by extending some basic facts about simple multicurves to the setting of multicurves with self-intersections.   It is well-known that the set of simple closed geodesics on $\Sigma$ are contained in a compact subset of $\Sigma$. A direct generalization of this argument also shows the following lemma that we state here for further reference:

\begin{lem}\label{compact}
For every $k$ there exists a compact set $K\subset \Sigma$ such that every geodesic multicurve $\gamma\subset \Sigma$ with at most $k$ self-intersections is contained in $K$.\qed
\end{lem}

Note that Lemma \ref{compact} implies that any sequence $(\gamma_n)$ of multicurves in $\Sigma$ with $\iota(\gamma,\gamma)= k$ has a Hausdorff convergent subsequence.
\medskip

Recall that a lamination is a compact subset of $\Sigma$ which is foliated by simple geodesics, and recall that the Hausdorff limits of sequences of simple multicurves are laminations (see \cite{Casson-Bleiler} for basic facts and definitions about laminations). Similarly, the Hausdorff limit $\lambda$ of a sequence $(\gamma_n)$ of multicurves with self-intersections is a union of geodesics but, naturally, they can intersect. However, if the multicurves $\gamma_n$ have a bounded number of intersections, then the non-simple leaves in $\lambda$ are finite and isolated in the following sense:  

\begin{lem}\label{lamination}
Given $k$, let $(\gamma_n)$ be a sequence in $\CS_k$, and suppose that it converges to some $\lambda\subset \Sigma$ in the Hausdorff topology as $n\to\infty$. Then $\lambda=\lambda_0\cup \mathcal{A}$ where $\lambda_0$ is a lamination and $\mathcal{A}$ is the union of finitely many geodesics $g_1,\dots,g_r$ such that for each $i$ there exists $j$, possibly $j=i$, such that $g_i$ and $g_j$ intersect transversely. 
\end{lem}

\begin{proof}
As we mentioned above, the Hausdorff limit $\lambda$ is a union of (images of) geodesics. We call a point $x\in\lambda$ singular if there exist geodesics $g, \bar{g}:\BR\to\Sigma$ parametrized by arc length with $g(\BR),\bar g(\BR)\subset\lambda$ and such that $g(0)=\bar{g}(0)=x$ but $\bar{g}'(0)\neq\pm g'(0)$. Say that a geodesic with image in $\lambda$ is singular if it goes through a singular point and note that the same argument used to prove that the closure of a set of simple disjoint geodesics is a lamination (see Lemma 3.2 in \cite{Casson-Bleiler}) shows that 
\begin{itemize}
\item any geodesic with image in $\lambda$ and sufficiently close to a singular geodesic is singular as well,
\item $\lambda$ has at most $k$ singular points and hence, up to reparametrization, at most $2k$ singular geodesics, and
\item the union $\CA$ of the images of the singular geodesics is open in its closure.
\end{itemize}
The last point implies that $\lambda_0=\lambda\setminus\CA$ is closed and hence compact. Since by construction $\lambda_0$ does not contain any singular points, it is a lamination, as we needed to prove.
\end{proof}

Equipped with this lemma we can construct radallas that carry all but finitely many multicurves with bounded number of self-intersections. Moreover, we can do it in such a way that after fixing some arbitrary $\epsilon>0$, each radalla is $\epsilon$-geodesic and only finitely many radallas are needed. The basic idea is to construct, for each Hausdorff limit $\lambda=\lambda_0\cup\mathcal{A}$ as in Lemma \ref{lamination}, a train-track carrying the lamination $\lambda_0$ in the usual way, and then add an edge for each of the finite leaves in $\mathcal{A}$. 

\begin{lem}\label{finite}
For any $k$ and $\epsilon>0$ there exists a finite collection of $\epsilon$-geodesic radallas $\hat{\tau}_1, \hat{\tau}_2, \ldots, \hat{\tau}_n$ such that for all but finitely many $\gamma\in\CS_k$ there is $i$ with $\gamma\prec\hat{\tau}_i$.
\end{lem}

Note that by definition the curves carried by an $\epsilon$-geodesic radalla are very long when $\epsilon$ is small -- this is because we assumed not only that the edges have small geodesic curvature but also that they are very long. In particular, once can think of the finite exceptional collection in Lemma \ref{finite} as the set of short curves in $\CS_k$.

\begin{proof}
Let $\overline{\CS_k}$ be the closure of $\CS_k$ in the set of all compact subsets of $\Sigma$ with respect to the Hausdorff topology. Lemma \ref{compact} implies that $\overline{\CS_k}$ is itself compact. Let $\lambda\in\overline{\CS_k}\setminus\CS_k$ be an accumulation point. Fixing an arbitrary $\epsilon>0$, we will construct an $\epsilon$-geodesic radalla $\hat{\tau}_{\lambda}$ that carries all $\gamma\in\CS_k$ that are sufficiently (Hausdorff) close to $\lambda$.  By Lemma \ref{lamination}, $\lambda=\lambda_0\cup \mathcal{A}$ where $\lambda_0$ is a lamination and $\mathcal{A}$ is the finite set of singular leaves. Let $\epsilon'>0$ be very small and take a regular $\epsilon'$-neighborhood of $\lambda$ and denote it by $\mathcal{N}(\lambda)$. For each of the finitely many singular points in $\mathcal{A}$ take a $2\epsilon'$-ball around it and let $B$ be the union of these balls. Then $\mathcal{N}(\lambda)\setminus B$ admits a foliation transversal to $\lambda$. Let $\gamma$ be a geodesic multicurve which is a distance less than $\epsilon'$ from $\lambda$. This curve can be isotoped to a curve which remains transverse to the foliation in  $\mathcal{N}(\lambda)\setminus B$ and which follows the leaves of $\mathcal{A}$ inside each ball in $B$. After isotoping each such curve in this manner, collapse $\mathcal{N}(\lambda)$ along the transverse foliation (in $\mathcal{N}(\lambda)\setminus B$) and to the leaves of $\mathcal{A}$ inside $B$. This results in a radalla $\hat{\tau}_{\lambda}$, where the associated train-track is the image of $\lambda_0$ under this collapse. By choosing $\epsilon'$ small enough, we can assume $\hat{\tau}_{\lambda}$ is $\epsilon$-geodesic. 

Note that, by construction, every multicurve in $\CS_k$ contained in some open neighborhood in $\overline{\CS_k}$ of $\lambda$ is carried by the $\epsilon$-geodesic radalla $\hat{\tau}_{\lambda}$. Compactness of $\overline{\CS_k}$ implies that finitely many such open sets cover a neighborhood of $\overline{\CS_k}\setminus\CS_k$ in $\overline{\CS_k}$. The claim follows.
\end{proof}

\subsection{Generic curves} 
Suppose $\Sigma$ has genus $g$ and $r$ punctures. As in \eqref{eqdefset}, let $\CS_k$ be the set of all multicurves in $\Sigma$ with $k$ self-intersections. 

\begin{defi*}
Let $\CZ\subset\CS_k$ be arbitrary. A subset $\CZ'\subset\CZ$ is \emph{negligible} if
\[ 
\lim_{L\to\infty}\frac{1}{L^{6g-6+2r}}|\{\gamma\in\CZ', \, \ell_\Sigma(\gamma)\leq L\}| = 0 
\]
The complement $\CZ\setminus\CZ'$ of a negligible set $\CZ'$ is said to be \emph{generic in $\CZ$}. 
\end{defi*}

If $\CZ'$ is generic in $\CZ$, and if the ambient set $\CZ$ is understood from the context, then we just say that $\CZ'$ is {\em generic}. 

\begin{bem}
Note that $\CZ'\subset\CZ$ could be negligible and generic at the same time. Note also that the image under a mapping class of a negligible set is also negligible. 
\end{bem}

In this section we prove that the set of all $\gamma\in\CS_k$ which {\em fill} an almost geodesic {\em maximal} radalla is generic. Here we say that a multicurve $\gamma$ carried by a radalla $\hat\tau$ {\em fills} if the corresponding vector of weights $\omega_\gamma$ is positive, meaning that $\omega_\gamma(e)>0$ for every $e\in E(\hat\tau)$. If $\gamma$ fills $\hat\tau$ we write $\gamma\fills\hat\tau$. A radalla $\hat{\tau}=(\hat{\tau}, \tau, \phi : \hat{\tau} \hookrightarrow \Sigma)$ is \emph{maximal} if $\phi(\tau)$ is a maximal recurrent train-track. Recall that a train-track $\tau$ is {\em recurrent} if it is filled by some multicurve. A recurrent train-track is maximal if it is not properly contained in any other recurrent train-track. 

\begin{bem}
If $\Sigma$ is not a once-punctured torus, then the complementary regions of a maximal train-track are just triangles and once-punctured monogons (see section \ref{sec-torus} for a discussion of the case of the once punctured torus). It follows that if $\Sigma$ is closed, then all complementary regions of a maximal recurrent train-track are triangles, which in turn implies that a maximal radalla is nothing but a maximal recurrent train-track. 
\end{bem}

We can now state precisely the main goal of this section:

\begin{prop}\label{generic}
For any $\epsilon>0$, the set 
\[
\CS_k^{\epsilon}= \{\gamma\in\CS_k\,|\, \gamma\fills\hat{\tau} \text{ for some maximal } \epsilon\text{-geodesic radalla } \hat{\tau} \}
\]
is a generic subset of $\CS_k$.
\end{prop}

Proposition \ref{generic} is going to follow easily from Lemma \ref{finite} once we determine ``how many" multicurves in $\CS_k$ are carried by each radalla: 

\begin{lem}\label{poly-bound}
Let $\Sigma$ be a hyperbolic surface of genus $g$ with $r$ punctures, let $(\hat\tau,\tau,\phi:\hat\tau\looparrowright\Sigma)$ be a radalla, and suppose that the underlying train-track $\tau$ is trivalent and is not properly contained in any other $\tau'\subset\hat\tau$ such that $(\hat\tau,\tau',\phi:\hat\tau\looparrowright\Sigma)$ is a radalla. Then we have
$$C=\limsup_{L\to\infty}\frac 1{L^{6g-6+2r}}\left\vert\left\{
\begin{array}{c}
\gamma\ \hbox{multicurve},\ \ \gamma\fills\hat\tau\\
\iota(\gamma,\gamma)= k,\ \ \ell_{\Sigma}(\gamma)\le L
\end{array}
\right\}\right\vert<\infty$$
for every $k\in\BN$. Moreover $C=0$ unless $\phi(\tau)$ is a maximal recurrent train-track.
\end{lem}

\begin{proof}
Given a multicurve $\gamma\in\CS_k$ carried by $\hat\tau$ let $\ell_{\hat\tau}(\gamma)$ be the length of some, and hence any, immersion homotopic to $\gamma$ and of the form $\phi\circ\gamma'$ where 
$$\gamma':\BS^1\sqcup\dots\sqcup\BS^1\to\hat\tau.$$ 
As mentioned earlier, lifts to the universal cover of curves carried by a radalla are bi-lipschitz embeddings where the bi-lipschitz constant is uniform. It follows that it suffices to prove that 
$$C'=\limsup_{L\to\infty}\frac 1{L^{6g-6+2r}}\left\vert\left\{
\begin{array}{c}
\gamma\ \hbox{multicurve},\ \ \gamma\fills\hat\tau\\
\iota(\gamma,\gamma)= k,\ \ \ell_{\hat\tau}(\gamma)\le L
\end{array}
\right\}\right\vert<\infty$$
and that $C'=0$ unless $\phi(\tau)$ is a maximal recurrent train-track. Let $c$ be an upper bound for the length of the images $\phi(e)$ of the edges of $\hat\tau$ and note that for any multicurve $\gamma$ carried by $\hat\tau$ one has
$$\ell_{\hat\tau}(\gamma)\le c\cdot\Vert\omega_\gamma\Vert_1$$
where $\Vert\cdot\Vert_1$ stands for the $L^1$-norm. Letting $\CV$ denote the subset of $\BN^{E(\hat\tau)}$ consisting  of solutions of the switch equations corresponding to some filling multicurve $\gamma$ with $\iota(\gamma,\gamma)= k$ we have thus
$$\left\vert\left\{
\begin{array}{c}
\gamma\ \hbox{multicurve},\ \ \gamma\fills\hat\tau\\
\iota(\gamma,\gamma)= k,\ \ \ell_{\hat\tau}(\gamma)\le L
\end{array}
\right\}\right\vert\le
K\cdot \left\vert\left\{
\omega\in\CV,\ \ \Vert\omega\Vert_1\le c\cdot L\\
\right\}\right\vert$$
where $K$ is the constant provided by Proposition \ref{prop-cars}. 

Note that the assumption that the vectors $\omega\in\CV$ correspond to weights of filling multicurves implies that each entry $\omega(e)$ is positive. Now, if $e\in E(\hat\tau)\setminus E(\tau)$ is an edge which is not contained in the train-track part, then there is another (possibly identical) edge $e'$ with $\iota(e,e')\ge 1$. Recalling that 
$$\omega(e)\cdot\omega(e')\cdot\iota(e,e')\le\iota(\gamma,\gamma)$$
for any $\gamma$ with $\omega_\gamma=\omega$ we get that 
$$\omega(e)\le k$$
for every $\omega\in\CV$ and every edge $e$ in $\hat\tau\setminus\tau$. This implies that $\CV$ is contained in finitely many translates of the set $\CW_\BZ$ of integral points in the linear subspace  
$$\CW=\{\omega\in\BR^{E(\tau)}\vert\ \hbox{solution of the switch equations in}\ \tau\}\subset \BR^{E(\hat\tau)}$$ 
of solutions of the switch equations supported by the train-track $\tau$. It follows that there is some $C$ with
$$\left\vert\left\{
\omega\in\CV,\ \ \Vert\omega\Vert_1\le c\cdot L\\
\right\}\right\vert\le C\left\vert\left\{
\omega\in\CW_\BZ,\ \ \Vert\omega\Vert_1\le c\cdot L\\
\right\}\right\vert$$
The linear space $\CW$ is defined over $\BZ$, and this implies that the number of integer points in $\CW$ grows like a polynomial  of degree equal to its dimension $\dim_{\BR}(\CW)$. Since, as it is well-known, $\dim_\BR(\CW)\le 6g-6+2r$ with equality if and only if $\tau$ is a maximal recurrent train-track, the claim follows.
\end{proof}

We can now prove Proposition \ref{generic}:

\begin{proof}[Proof of Proposition \ref{generic}]
By Lemma \ref{finite} there are, for all $\epsilon$, finitely many $\epsilon$-geodesic radallas $\hat{\tau}_1, \ldots, \hat{\tau}_n$ which carry all but finitely many curves in $\CS_k$. Since each radalla contains only finitely many other radallas, we can assume, up to adding finitely many radallas to our list, that the radallas $\hat\tau_i$ satisfy the condition in Lemma \ref{poly-bound} and that for all but finitely many $\gamma\in\CS_k$ there is some $i$ with $\gamma\fills\hat\tau_i$. Now, it follows from Lemma \ref{poly-bound} that the set of those multicurves carried and filling a non-maximal radalla is negligible. The claim follows because the finite union of negligible sets is negligible. 
\end{proof}

Note at this point that in fact it follows from Lemma \ref{poly-bound} and from the argument used in the proof of Proposition \ref{generic} that set $\CS_k$ has at most polynomial growth of degree $L^{6g-6+2r}$:
\[
\limsup_{L\to\infty}\frac 1{L^{6g-6+2r}}\left\vert\left\{
\gamma\in\CS_k\vert\ell_{\Sigma}(\gamma)\le L
\right\}\right\vert<\infty
\]
A lower bound of the same order of magnitude can be obtained just by adding crossings to the simple multicurves carried by some maximal recurrent train-track,  but it is anyways due to Sapir \cite{Sapir1,Sapir2}:

\begin{kor}\label{poly-growth}
Let $\Sigma$ be a hyperbolic surface of genus $g$ and $r$ punctures. Then there is $C\ge 1$ with
$$\frac 1C\le \frac 1{L^{6g-6+2r}}\left\vert\left\{
\gamma\in\CS_k\vert\ell_{\Sigma}(\gamma)\le L
\right\}\right\vert\le C$$
for all $L$ large enough.\qed
\end{kor}

\subsection{Angles}
We prove Theorem \ref{closed-small-angles} from the introduction next:

\begin{named}{Theorem \ref{closed-small-angles}}
Let $\Sigma$ be a closed hyperbolic surface of genus $g\geq 2$ and let $\measuredangle(\gamma)\in(0,\frac\pi2]$ denote the largest angle among the self-intersections of a multicurve $\gamma\subset\Sigma$. Then 
\[
\lim_{L\to\infty}\frac{1}{L^{6g-6}}\left\vert\left\{
\begin{array}{c}
\gamma\subset \Sigma \text{ multicurve},  \iota(\gamma,\gamma)= k, \\ \measuredangle(\gamma)\geq\delta,\, \ell_\Sigma(\gamma)\leq L
\end{array}
\right\}\right\vert=0
\]
for every $k$ and every $\delta>0$.
\end{named}

\begin{proof}
Given $\delta>0$ there is $\epsilon$ with $\measuredangle(\gamma)<\delta$ for every multicurve $\gamma$ carried by some $\epsilon$-geodesic train-track. In other words, the set 
\[
\mathcal{S}_{k,\measuredangle<\delta}=\{\gamma\in\CS_k\text{ with } \measuredangle(\gamma)<\delta\}
\]
contains the set 
\begin{equation}\label{eq-blablalbalba}
\{\gamma\in\CS_k\,|\, \gamma\fills\tau \text{ for some maximal } \epsilon\text{-geodesic train-track } \tau \}
\end{equation}
of all multicurves with $k$ self-intersections which fill some $\epsilon$-geodesic train-track. As noted earlier, the assumption that $\Sigma$ is closed implies that a maximal radalla is in fact a maximal recurrent train-track. It follows thus from Proposition \ref{generic} that the set \eqref{eq-blablalbalba} is generic in $\CS_k$. 
\end{proof}

Theorem \ref{closed-small-angles} fails if $\Sigma$ is not closed. We can in fact divide the self-intersections into two types which we refer to as \emph{small} and \emph{large}. Supposing hat $\Sigma$ is not a once punctured torus (see section \ref{sec-torus} for this case) and with notation as in Proposition \ref{generic}, let $\hat{\tau}=(\hat\tau,\tau,\phi:\hat\tau\looparrowright\Sigma)$ be a maximal $\epsilon$-geodesic radalla which is filled by some curve with $k$ self-intersections. Since $\hat\tau$ is maximal, and since we are assuming that $\Sigma$ is not a once-punctured torus, it follows that all the complementary regions of $\phi(\tau)$ are either triangles or once punctured monogons. If $\epsilon$ is small enough they are indeed almost ideal triangles and almost ideal once-punctured monogons. The ideal triangles cannot contain any additional leaves of $\phi(\hat\tau)$ but the punctured monogons can have one or several such leaves which then have self-intersections with relatively large angles. In fact, a simple computation in hyperbolic geometry yields that if $\epsilon$ is small enough then every intersection between leaves $e,e'$ of a maximal $\epsilon$-geodesic radalla happen at an angle greater than $\frac 1k$ whenever $k\ge\iota(e,e)$ and $k\ge\iota(e',e')$.

Moreover, as long as $\epsilon$ is chosen small enough we get as in the proof of Theorem \ref{closed-small-angles} that the self-intersection angles of a geodesic multicurve carried by an $\epsilon$-geodesic $\hat\tau$ are close to those of the representative in the radalla. In other words we get:

\begin{lem}\label{lem-lem-lem-idontknowwhat}
Suppose that $\Sigma$ is not a once-punctured torus. For every $k$ and $\delta$ positive there is $\epsilon$ such that if $\gamma\subset\Sigma$ is a multicurve with $\iota(\gamma,\gamma)=k$ and which is carried by a maximal $\epsilon$-geodesic radalla, then $\gamma$ has no self-intersection with angle in $[\delta,\frac 1k]$. \qed
\end{lem}

We will refer to intersections with angle larger than $\frac 1k$ as being {\em large}. The remaining ones are {\em small}. Recall that it follows from the argument leading to Lemma \ref{lem-lem-lem-idontknowwhat} that, as long as $\epsilon$ is small enough, large self-intersections of multicurves $\gamma\in\CS_k$ which are carried by a maximal $\epsilon$-geodesic radalla $\hat\tau$ correspond to intersection points of (possibly equal) edges of the radalla.


\subsection{The map} \label{sec-themap}
Still assuming that $\Sigma$ is not a once punctured torus, fix some $k$ and $\delta$ with $\delta\lll\frac 1k$ and fix once and for all $\epsilon$ positive but very small satisfying Lemma \ref{lem-lem-lem-idontknowwhat}. 

Our goal is to construct a map 
$$\pi_{\epsilon,k}:\CS_k^\epsilon\to\CM\CL_\BZ(\Sigma)$$
from the set $\CS_k^\epsilon$ of multicurves with $k$ self-intersections carried by a maximal $\epsilon$-geodesic radalla, to the set $\CM\CL_\BZ(\Sigma)=\CS_0$ of simple multicurves -- this map will be the key to relate the growth of the number of self-intersecting curves of some type to the growth of simple multicurves. To intuitively explain the construction, suppose that all the self-intersections of $\gamma\in\CS_k^\epsilon$ are small. In this case we let $\pi_{\epsilon,k}(\gamma)\in\CM\CL_\BZ(\Sigma)$ be the multicurve obtained from $\gamma$ by resolving the self-intersections in such a way that $\gamma$ and $\pi_{\epsilon,k}(\gamma)$ remain almost parallel.

Consider now the general case. Suppose that we are given $\gamma\in\CS_k^{\epsilon}$ and a maximal $\epsilon$-geodesic radalla $\hat\tau=(\hat{\tau}, \tau, \phi : \hat{\tau} \hookrightarrow \Sigma)$ with $\gamma\fills\hat\tau$. We start by associating to $\gamma$ a simple multicurve $\pi_{\hat\tau}(\gamma)$ with the help of the radalla. As before, let $E(\hat\tau)$ be the set of edges of $\hat{\tau}$ and $E(\tau)\subset E(\hat\tau)$ the set of edges of the associated train-track. Recall that we are assuming that $\Sigma$ is not a once-punctured torus. As we mentioned before, this implies that each edge in $E(\hat\tau)\setminus E(\tau)$ is contained in some punctured monogon. In particular, we can associate to each $e\in E(\hat\tau)\setminus E(\tau)$ a locally embedded loop $\tilde e\subset\tau$ in the train-track with the same endpoint as $e$ and whose image represents the boundary of the punctured monogon containing $e$ (see figure \ref{hand-picture1}). 

\begin{figure}[h]
\includegraphics[width=8cm, height=4cm]{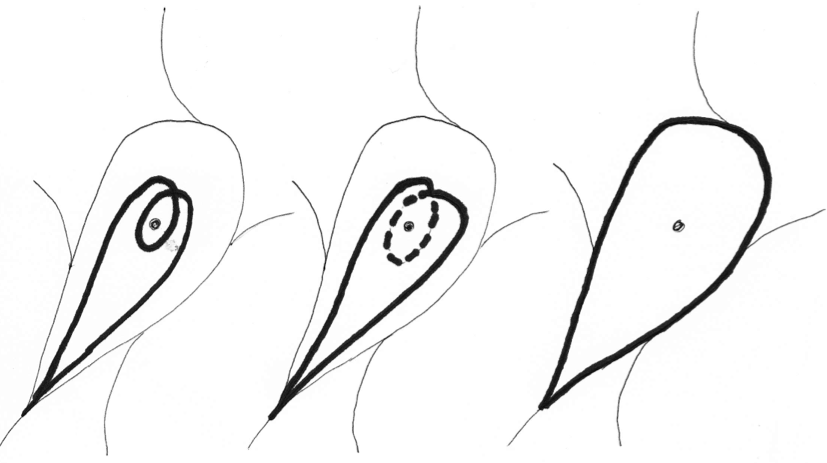}
\caption{The solid line on the left image is an edge $e\in E(\hat\tau)\setminus E(\tau)$ and on the right it represents the loop $\tilde e$. In the middle we see an intermediate step given by deleting loops around the puncture.}
\label{hand-picture1}
\end{figure}

We denote by $\omega_{\tilde e}\in\BR^{E(\tau)}$ the vector of weights of $\tilde e$ and define a linear map
\begin{equation}\label{eq:map-weights}
\pi_{\hat\tau}:\BR^{E(\hat\tau)}\to \BR^{E(\tau)}
\end{equation}
as the identity on $\BR^{E(\tau)}\subset\BR^{E(\hat\tau)}$ and as mapping the basis vector $e\in\BR^{E(\hat\tau)}$ to the vector $\omega_{\tilde e}\in\BR^{E(\tau)}$ for all $e\in E(\hat\tau)\setminus E(\tau)$ (see figure \ref{hand-picture4a}). 

Note that this map can also be viewed as being induced by a smooth map $\hat\tau\to\tau$ from the radalla to the train-track, where $\tau$ is fixed point-wise and where each additional edge $e\in E(\hat\tau)\setminus E(\tau)$ is mapped to the corresponding path $\tilde e\subset\tau$. An alternative description: each edge $e\in\hat{\tau}\setminus\tau$ has one or several loops around a puncture and we are cutting out such loops and homotopying the obtained segment into the train-track (compare with figure \ref{hand-picture1}).

\begin{figure}[h]
\includegraphics[width=9cm, height=5cm]{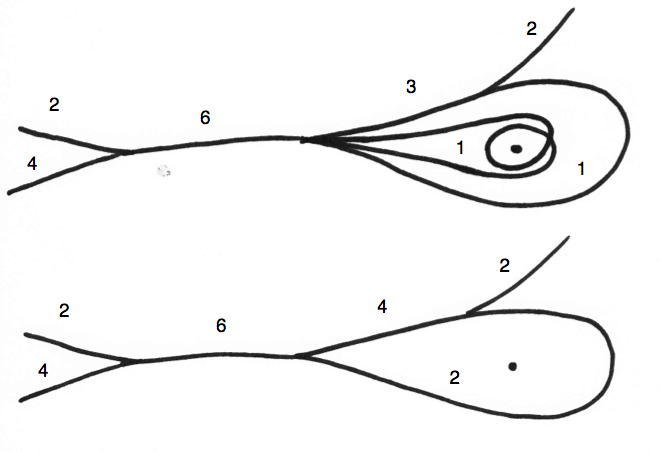}
\caption{The map $\pi_{\hat\tau}$: each number is the weight of the edge underneath. The upper picture is in the radalla $\phi(\hat\tau)$ and the lower in the train-track $\phi(\tau)$. The weights on the train-track are the images of the weights on the radalla under the map $\pi_{\hat\tau}$.}\label{hand-picture4a}
\end{figure}

Anyways, we obtain thus a way to associate to each $\gamma\fills\hat\tau$ a simple multicurve carried by the underlying train-track $\tau$ as follows: consider the vector $\omega_\gamma\in\BR^{E(\hat\tau)}$ of weights associated to $\gamma$, apply $\pi_{\hat\tau}$ as in \eqref{eq:map-weights} and let $\pi_{\hat\tau}(\gamma)\prec\tau$ be the simple multicurve with 
$$\omega_{\pi_{\hat\tau}(\gamma)}=\pi_{\hat\tau}(\omega_{\gamma}).$$
Next we give a different description of $\pi_{\hat\tau}(\gamma)$ which does not use the radalla $\hat\tau$, just its existence.

\begin{figure}[h]
\includegraphics[angle=270,origin=c,width=10cm, height=5cm]{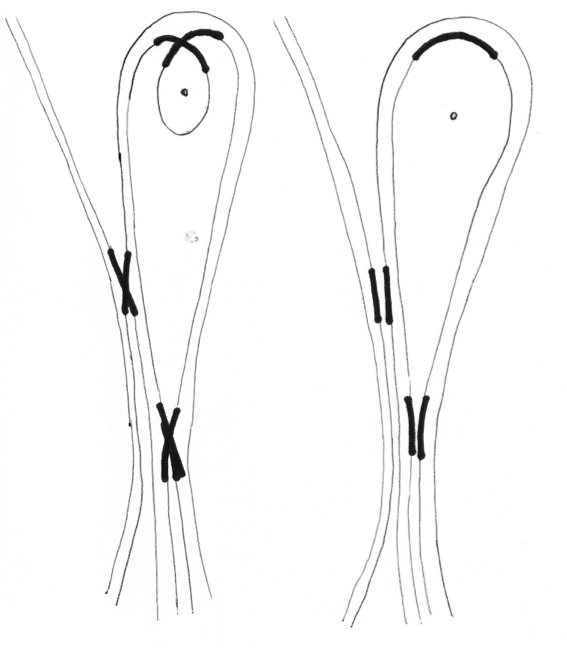}
\caption{The effect of the map $\pi_\epsilon$ on a curve. In the picture, two small intersections and one large intersection have been resolved.}
\label{hand-figure3a}
\end{figure}

Recall that by Lemma \ref{lem-lem-lem-idontknowwhat}, and by the choice of $\epsilon$, we can divide  the intersection points of any $\gamma\fills\hat{\tau}$ into two types: small and large. We resolve small intersections using the obvious local move: replace two intersecting segments by two which are disjoint to each other and almost parallel to the original ones (see figure \ref{hand-figure3a}). To resolve large intersections note that each one of them corresponds to a loop around a puncture -- we just remove all such loops (see again figure \ref{hand-figure3a}). Proceeding like this for all intersection points we obtain a simple multicurve $\gamma_0$ which is in fact homotopic to the simple multicurve that we get applying $\pi_{\hat\tau}$ (compare figure \ref{hand-picture4a} and figure \ref{hand-figure3a}). It follows thus that $\pi_{\epsilon,k}(\gamma)\stackrel{\tiny\text{def}}=\pi_{\hat\tau}(\gamma)$ does {\em not} depend on the radalla $\hat\tau$, meaning that we have a well-defined map
\begin{equation}\label{eq-themap}
\pi_{\epsilon,k}: \CS_k^{\epsilon}\to \mathcal{ML}_{\mathbb{Z}}
\end{equation}

\begin{bem}
We stress that the only requirement that we put on $\epsilon$ is that it satisfies Lemma \ref{lem-lem-lem-idontknowwhat} for some $\delta$ small enough, say for $\delta=\frac 1{100k}$. In fact, we will not need to reduce $\epsilon$ later on. However, the reader might find it reassuring to observe that if $\epsilon'\leq\epsilon$, then $\pi_{\epsilon,k}\vert_{\CS_k^{\epsilon'}}=\pi_{\epsilon',k}$. 
\end{bem}

We establish some properties of the map \eqref{eq-themap}. Suppose that $\hat\tau=(\hat{\tau}, \tau, \phi : \hat{\tau} \hookrightarrow \Sigma)$ is a radalla and that $\eta\subset\Sigma$ is a simple closed curve {\em transversal} to $\hat\tau$, by which we mean that both are in general position with respect to each other and there are no proper smooth arcs $I\subset\eta$ and $J\subset\hat\tau$ such that $I$ and $\phi(J)$ are isotopic to each other relative to their endpoints. Suppose also that $U$ is a complementary region of the train-track $\phi(\tau)$, and that $U$ is a punctured monogon. Then each component of $\eta\cap U$ meets the image $\phi(e)$ exactly twice for each edge $e\in E(\hat\tau)\setminus E(\tau)$ with $\phi(e)\subset U$. It follows thus directly from the definition of $\pi_{\hat\tau}$ that
$$\sum_{e\in E(\hat\tau)}\vert\eta\cap\phi(e)\vert\cdot\omega(e)=\sum_{e\in E(\tau)}\vert\eta\cap\phi(e)\vert\cdot(\pi_{\hat\tau}\omega)(e)$$
for every $\omega\in\BR^{E(\hat\tau)}$.
Applying this to $\omega=\omega_\gamma$ for $\gamma\fills\hat\tau$ we get that
\begin{align*}
\iota(\eta,\gamma)
&=\sum_{e\in E(\hat\tau)}\vert\eta\cap\phi(e)\vert\cdot\omega(e)\\
&=\sum_{e\in E(\tau)}\vert\eta\cap\phi(e)\vert\cdot(\pi_{\hat\tau}\omega)(e)\\
&=\iota(\eta,\pi_{\epsilon,k}(\gamma)
\end{align*}
where the first (resp. last) equality holds because $\eta$ and $\gamma$ (resp. $\pi_{\epsilon,k}(\gamma)$) are transversal. We record this fact:

\begin{lem}\label{transverse-inter}
Suppose that $\hat\tau$ is an $\epsilon$-geodesic radalla and that $\eta\subset\Sigma$ is a simple closed geodesic transversal to $\hat\tau$. Then we have
$$\iota(\eta,\pi_{\epsilon,k}(\gamma))=\iota(\eta,\gamma)$$
for every $\gamma\in\CS_k^\epsilon$ with $\gamma\fills\hat\tau$. \qed
\end{lem}

In some sense Lemma \ref{transverse-inter} asserts that $\gamma$ and $\pi_{\epsilon,k}(\gamma)$ are close to each other. Compare with Fact \ref{doesnotmove} in the next section.

Observe that the map \eqref{eq-themap} maps curves of some length to curves of roughly the same length. By itself, this already implies that $\pi_{\epsilon,k}$ is finite-to-one. The main content of the following proposition is that it is actually bounded-to-one:

\begin{prop}\label{prop-key-map}
There is $\kappa$ with $\vert\pi^{-1}_{\epsilon,k}(\gamma)\vert\leq\kappa$ for all $\gamma\in\mathcal{ML}_{\mathbb{Z}}$. On the other hand, if $\tau$ is an $\epsilon$-geodesic maximal train-track, and $\phi\in\Map(\Sigma)$ is a mapping class with $\phi(\tau)\prec\tau$, then one has $\vert\pi_{\epsilon,k}^{-1}(\phi(\gamma))\vert\ge\vert\pi_{\epsilon,k}^{-1}(\gamma)\vert$ for every simple multicurve $\gamma\fills\tau$. 
\end{prop}

\begin{proof}
Note that the image of $\pi_{\epsilon,k}$ consists of simple multicurves which are carried and fill some maximal $\epsilon$-geodesic train-track. Given such a train-track $\tau$ and a simple multicurve $\gamma\fills\tau$ then $\pi_{\epsilon,k}^{-1}(\gamma)$ consists of elements of $\CS_k^\epsilon$ which are carried by some radalla $\hat\tau$ extending $\tau$. Recall also that maximality of $\tau$ implies that the leaves $e$ in $\hat\tau\setminus\tau$ are contained in once-punctured monogons -- if $r$ is the number of cusps, there are $r$ such punctured monogons. Each one of these leaves in a punctured monogon intersects itself and in fact there are, up to isotopy, only $k$ leaves with at most $k$ self-intersections. All this implies that there are at most $r2^k+1$ radallas $\hat\tau$ extending $\tau$ and which are filled by some curve in $\CS_k$. 

Hence, to prove the first claim it suffices to bound, for each one of these radallas, the number of multicurves $\gamma'\in\CS_k$ with $\gamma'\fills\hat\tau$ and $\pi_{\epsilon,k}(\gamma')=\gamma$. Note that, by construction of the map, the weights $\omega_{\gamma'}$ of each such curve $\gamma'$ belong to the preimage of the map in \eqref{eq:map-weights}. Since $\iota(\gamma',\gamma')=k$ and since $\iota(e,e)\ge 1$ for each $e\in E(\hat\tau)\setminus E(\tau)$ we get that $\omega_{\gamma'}(e)\le k$ for any such edge. It follows that $\omega_{\gamma'}$ belongs to a collection of at most $k^2$ vectors in $\BR^{E(\hat\tau)}$. By Proposition \ref{prop-cars} we know that there is some $K$ such that each vector in $E(\hat\tau)$ corresponds to at most $K$ curves carried by $\hat\tau$ and with $k$ self-intersections. Altogether we get that $\pi_{\epsilon,k}^{-1}(\gamma)$ consists of at most $\kappa=(r2^k+1)k^2K$ elements. We have proved the first claim.

To prove the second claim vote that it suffices to show that 
$$\pi_{\epsilon,k}(\phi(\gamma'))=\phi(\pi_{\epsilon,k}(\gamma'))$$
for every mapping class $\phi\in\Map(\Sigma)$ with $\phi(\tau)\prec\tau$, and every $\gamma'\in\pi_{\epsilon,k}^{-1}(\gamma)$ for each $\gamma\fills\tau$. Any such $\gamma'$ is carried by some radalla $\hat\tau$ extending the train-track $\tau$. Since $\phi(\tau)\prec\tau$, we get that $\phi(\hat\tau)$ is carried by some radalla $\hat\tau'$ extending $\tau$. Note that $\hat\tau'$ can be isotoped to be $\epsilon$-geodesic. Since we can compute $\pi_{\epsilon,k}(\gamma')$ and $\pi_{\epsilon,k}(\phi(\gamma'))$ using the radallas $\hat\tau$ and $\hat\tau'$, we get thus that $\pi_{\epsilon,k}(\phi(\gamma'))=\phi(\pi_{\epsilon,k}(\gamma'))=\phi(\gamma)$, as we wanted to prove. This concludes the proof of Proposition \ref{prop-key-map}.
\end{proof}

Before moving on, note that the bound we gave in the proof of the first claim of Proposition \ref{prop-key-map} is rather brutal. In fact, in the absence of cusps we can replace the bound by the constant $K$ from Proposition \ref{prop-cars}. Moreover, in the cases when we computed $K$ explicitly (i.e. $k=1,2$) after the proof of Proposition \ref{prop-cars} one can give a formula, even in the presence of cusps, for the number of preimages. To see this, let $\tau$ be a maximal $\epsilon$-geodesic train-track and $\gamma$ a simple multicurve carried by and filling $\tau$. Suppose moreover that $\omega_\gamma(e)\ge 10$ for every edge $e\in E(\tau)$. Consider the case of $k=1$ and let $\gamma' \in\pi_{\epsilon,1}^{-1}(\gamma)\subset \CS^{\epsilon}_{1}$. Then $\gamma'\fills\hat\tau$ for some radalla $\tau'$ extending $\tau$. Since $\tau$ is maximal, any leaf of $\hat\tau\setminus\tau$ is contained in one of the $r$ once-punctured monogons. Since $\gamma'$ fills $\hat\tau$ and has exactly one self-intersection it follows that there is at most one edge in $\hat\tau\setminus\tau$ and hence there are (up to isotopy) exactly $r+1$ such radallas. Now, if $\gamma' \fills\hat\tau$ where $\hat\tau\neq\tau$ then (the homotopy class of) $\gamma'$ is uniquely determined by $\hat\tau$. If instead $\hat\tau=\tau$ there are, as explained in the discussion following the proof of Proposition \ref{prop-cars}, $\frac{V}{2}$ choices for $\gamma'$ where $V=12g-12+4r$ is the number of vertices of the maximal train-track $\tau$. Hence we have $\vert\pi_{\epsilon,1}^{-1}(\gamma)\vert=6g-6+3r$. A similar computation can be made for the case $k=2$ and we record the resulting count here, but leave the proof to the reader. 



\begin{lem}\label{countk12}
Suppose that $\Sigma$ has genus $g$ and $r$ punctures and is not homeomorphic to a once punctured torus. For some $\epsilon$ small enough, let $\tau$ be a maximal recurrent $\epsilon$-geodesic train-track and $\gamma\prec\tau$ a simple multicurve which traverses each edge of $\tau$ at least 10 times, that is $\omega_\gamma(e)\ge 10$ for all $e\in E(\tau)$. Then we have
\begin{itemize}
\item $\vert\pi_{\epsilon,1}^{-1}(\gamma)\vert=6g-6+3r$.
\item $\vert\pi_{\epsilon,2}^{-1}(\gamma)\vert=\frac92\big((2g+r)(2g+r-3)+2\big)$.\qed
\end{itemize}
\end{lem}


\section{}\label{sec-measures}

Suppose that $\Sigma$ is a hyperbolic surface other than a punctured torus and fix some natural number $k$ -- the case of the torus with be discussed in section \ref{sec-torus}. We also fix $\gamma_0\in\CS_k$, i.e.~a multicurve in $\Sigma$ with $\iota(\gamma_0,\gamma_0)=k$, and let 
$$\CS_{\gamma_0}=\Map(\Sigma)\cdot\gamma_0\subset\CS_k$$ 
be the set of curves of type $\gamma_0$. In this section we associate to this set and to every $L>0$ a measure $\nu_{\gamma_0}^L$ on the space $\CC(\Sigma)$ of {\em currents} on $\Sigma$. We will show that every accumulation point of $\nu_{\gamma_0}^L$ when $L\to\infty$ is a multiple of the Thurston measure $\mu_{\Thu}$ on the subspace $\CM\CL(\Sigma)$ of $\CC(\Sigma)$ consisting of measured laminations. We will moreover prove that the existence of an actual limit is equivalent to the existence of the limit \eqref{eq1}. We begin by recalling a few facts on currents, measured laminations, and the Thurston measure. Before doing so we refer to \cite{Casson-Bleiler,Thurston-notes,FLP} for basic facts on measured laminations (or equivalently, measured foliations) and their relation to train-tracks, and to \cite{Bonahon86,Bonahon88,Bonahon-currents-groups,Otal} for basic facts on currents and their relation to laminations. Both laminations and currents are treated in the extremely readable paper \cite{Javi-Cris}.

\subsection{Currents}
The total space $PT\Sigma$ of the projective tangent bundle of $\Sigma$ has a 1-dimensional foliation, the {\em geodesic foliation}, whose leaves are the traces of geodesics in $\Sigma$. This foliation is intrinsic in the sense that any homeomorphism $\Sigma\to\Sigma'$ between hyperbolic (or even just negatively curved) surfaces induces a homeomorphism $PT\Sigma\to PT\Sigma'$ which maps the geodesic foliation on $PT\Sigma$  to the geodesic foliation on $PT\Sigma'$. This is basically the well-known fact that geodesic flows on homeomorphic negatively curved manifolds of finite volume are orbit equivalent to each other. 

A {\em geodesic current} is a Radon transverse measure to the geodesic foliation. Equivalently a geodesic current is a $\pi_1(\Sigma)$-invariant Radon measure on the space $\CG(\tilde\Sigma)$ of geodesic in the universal cover $\tilde\Sigma$ of $\Sigma$. Also,  geodesic currents are in one-to-one correspondence with Radon measures invariant under both the geodesic flow and the geodesic flip. Recall that a Borel measure is Radon if it is locally finite and inner regular.

We denote the space of all geodesic currents endowed with the weak-*-topology by $\CC(\Sigma)$. If $\Sigma\to\Sigma'$ is a homotopy equivalence, then the map $\CC(\Sigma)\to\CC(\Sigma')$ between the spaces of currents induced by the foliation preserving homeomorphism $PT\Sigma\to PT\Sigma'$ is also a homeomorphism. In particular, the mapping class group $\Map(\Sigma)$ acts on $\CC(\Sigma)$ by homeomorphisms.

If $K\subset PT\Sigma$ is a compact set invariant under the geodesic flow, i.e. $K$ is saturated with respect to the geodesic foliation, let
$$\CC_K(\Sigma)=\{\lambda\in\CC(\Sigma)\vert\lambda(PT\Sigma\setminus K)=0\}$$ 
denote the set of currents (considered as measures invariant under the geodesic flow) supported by $K$. We will be interested in currents supported by a compact set because by Lemma \ref{compact} all curves in $\CS_{\gamma_0}$ live in a fixed compact set. On the other hand, currents such as the Liouville current $\lambda_{\Sigma}$ associated to the hyperbolic metric of $\Sigma$ belongs to some $\CC_K(\Sigma)$ only if $\Sigma$ itself is compact.

Every primitive curve (and in fact, every multicurve) in $\Sigma$ can be considered as a current: namely the Dirac measure centred at the said curve. In this way we can see the set $\CS=\CS(\Sigma)$ of all multicurves in $\Sigma$ as a subset of the space $\CC(\Sigma)$ of currents. In fact, it is known that the set $\BR_+\CS$ of all {\em weighted multicurves} is dense in $\CC(\Sigma)$. Moreover, when $\Sigma$ is closed, then the geometric intersection number $\CS\times\CS\to\BN$ extends uniquely to a continuous symmetric map $\CC(\Sigma)\times\CC(\Sigma)\to\BR_+$. In the presence of cusps things are more complicated, for instance because the map might take the value $\infty$. Anyways, the standard argument \cite{Bonahon86} proves that for every compact set $K\subset PT\Sigma$ invariant under the geodesic flow we have that\begin{equation}\label{intersection-form}
\iota:\CC(\Sigma)\times\CC_K(\Sigma)\to\BR_+
\end{equation}
takes finite values and is continuous. Moreover, $\iota(a\cdot\lambda,b\cdot\mu)=ab\cdot\iota(\lambda,\mu)$ whenever $a,b\in\BR_+$ are positive reals and $\lambda\in\CC(\Sigma)$ and $\mu\in\CC_K(\Sigma)$ are currents. The map \eqref{intersection-form} is called the {\em intersection form}. Note that if $\Sigma\to\Sigma'$ is again a homotopy equivalence between hyperbolic surfaces, then the homeomorphism $\CC(\Sigma)\to\CC(\Sigma')$ commutes with the scaling action of $\BR_+$ and with the intersection form. In particular, the intersection form \eqref{intersection-form} is invariant under the mapping class group of the surface.

We list a few facts on the intersection form \eqref{intersection-form}:
\begin{itemize}
\item If $\lambda_\Sigma$ is the Liouville current of the hyperbolic metric on $\Sigma$ then
$$\iota(\lambda_\Sigma,\gamma)=\ell_\Sigma(\gamma)$$ 
for every multicurve $\gamma\subset\Sigma$.
\item A current $\mu$ is {\em filling} if every geodesic in $\Sigma$ is transversally intersected by some geodesic in the support of $\mu$. For instance, if $\gamma\in\CS$ is a filling multicurve in the sense that it cuts $\Sigma$ into polygons and punctured monogons, then $\gamma$ is also filling as a current. The Liouville current $\lambda_\Sigma$ is also filling. 
\item If $\lambda$ is a filling current and $K\subset PT\Sigma$ is compact, then the set $\{\mu\in\CC_K(\Sigma)\vert\iota(\lambda,\mu)\le 1\}$ is compact in $\CC_K(\Sigma)$.
\item For any compact set $K\subset PT\Sigma$ invariant under the geodesic flow, and for any two filling currents $\mu,\lambda\in\CC(\Sigma)$, there is some $L$ with
$$\frac 1L\iota(\mu,\eta)\le\iota(\lambda,\eta)\le L\iota(\mu,\eta)$$
for every $\eta\in\CC_K(\Sigma)$.
\end{itemize}
See the references given earlier for proofs, keeping in mind that we are restricting \eqref{intersection-form} on the second factor to compactly supported currents. Under this restriction, the proofs are exactly the same as in the case that $\Sigma$ is closed.

As above, suppose that $\gamma_0\in\CS_k$ is a multicurve in $\Sigma$ with $\iota(\gamma_0,\gamma_0)=k$, and let 
$$\CS_{\gamma_0}=\Map(\Sigma)\cdot\gamma_0\subset\CS_k$$ 
be the set of curves of type $\gamma_0$. Consider the elements of $\CS_{\gamma_0}$ as currents on $\Sigma$ and let $K\subset PT\Sigma$ be a compact set, given by Lemma \ref{compact}, such that 
\begin{equation}\label{curvesascurrents}
\CS_{\gamma_0}\subset\CC_K(\Sigma).
\end{equation}
For $L>0$ we consider the measure
\begin{equation}\label{eq-measure-nu}
\nu_{\gamma_0}^L=\frac 1{L^{6g-6+2r}}\sum_{\gamma\in\CS_{\gamma_0}}\delta_{\frac 1L\gamma}
\end{equation}
on $\CC_K(\Sigma)\subset\CC(\Sigma)$, where $\delta_x$ stands for the Dirac measure centered at $x$. The goal of the sequel is to study the behavior of the measures $\nu_{\gamma_0}^L$ when $L$ tends to $\infty$. Among other things we will prove that every accumulation point of $\nu_{\gamma_0}^L$ when $L\to\infty$ is a multiple of the Thurston measure on $\mu_{\Thu}$ on the subspace of $\CC(\Sigma)$ consisting of measured laminations.

\subsection{Measured laminations}
Recall that a measured lamination is a lamination endowed with a transverse measure of full support. As such, a measured lamination is also a current. In fact, a current $\lambda\in\CC(\Sigma)$ is a measured lamination if and only if $\iota(\lambda,\lambda)=0$. Noting that laminations are contained in the compact set provided by Lemma \ref{compact}, we can see the space $\CM\CL(\Sigma)$ of all measured laminations on $\Sigma$ is thus a subset of $\CC_K(\Sigma)$ with the same $K$ as in \eqref{curvesascurrents}:
\begin{equation}\label{curvesascurrents2}
\CM\CL(\Sigma)\subset\CC_K(\Sigma).
\end{equation}
Being a subset of $\CC(\Sigma)$, the space $\CM\CL(\Sigma)$ of measured laminations has an induced topology. In fact, Thurston proved that $\CM\CL(\Sigma)$ is homeomorphic to $\BR^{6g-6+2r}$. Moreover, similar to the classical Fenchel-Nielsen coordinates, there is a finite collection (see for instance \cite{chefin}) of simple curves $\eta_1,\dots,\eta_s$ such that
\begin{equation}\label{FNML}
\text{if }\iota(\lambda,\eta_i)=\iota(\mu,\eta_i)\text{ for all }i\text{ then }\lambda=\mu
\end{equation}
for any two $\lambda,\mu\in\CM\CL(\Sigma)$. Note that the collection of curves $\eta_1,\dots,\eta_s$ is very far from being unique -- for instance, transforming this collection via a mapping class we get a new collection with the same property.

The space $\CM\CL(\Sigma)$ has not only a natural topology, but also a compatible mapping class group invariant integral PL-manifold structure. Here, integral means that the change of charts are given by linear transformations with integral coefficients. In particular, the set of integral points in $\CM\CL(\Sigma)$ is well-defined. In fact, for every recurrent train-track $\tau$ in $\Sigma$ one has the simplex in $\CM\CL(\Sigma)$ consisting of all measured laminations carried by $\tau$ -- the integral points in this simplex correspond simply to the integral solutions of the weight equations. It follows that the set of integral points of $\CM\CL(\Sigma)$ is just the set $\CM\CL_\BZ(\Sigma)$ of simple multicurves in $\Sigma$. Note before going further that the set of measured laminations carried by a train-track $\tau$ is full-dimensional if and only if $\tau$ is a maximal recurrent train-track.

The PL-manifold $\CM\CL(\Sigma)$ is in fact endowed with a mapping class group invariant symplectic structure \cite{Penner-Harer} and hence with a mapping class group invariant measure in the Lebesgue class. This measure is the so-called {\em Thurston measure} $\mu_{\Thu}$. It is an infinite but locally finite measure, positive on non-empty open sets, and satisfying
$$\mu_{\Thu}(L\cdot U)=L^{6g-6+2r}\mu_{\Thu}(U)$$ 
for all $U\subset\CM\CL(\Sigma)$ and $L>0$. Note that this implies that $\mu_{\Thu}(A)=0$ if $A\cap L\cdot A=\emptyset$ for all $L>0$. In particular,
$$\mu_{\Thu}(\{\lambda\in\CM\CL(\Sigma)\vert\iota(\lambda_0,\lambda)=1\})=0$$
for every filling current $\lambda_0\in\CC(\Sigma)$.

On charts, the measure $\mu_{\Thu}$ is just the standard Lebesgue measure and the integral points of $\CM\CL(\Sigma)$ are just points in the integral lattice. Hence we get, under weak assumptions on $U$, that the Thurston measure of a set $U$ can be computed by counting the integral points in $L\cdot U$, dividing by the appropriate power of $L$, and letting $L$ go to $\infty$. In a more succinct way
$$\mu_{\Thu}=\lim_{L\to\infty}\frac 1{L^{6g-6+2r}}\sum_{\gamma\in\CM\CL_\BZ}\delta_{\frac 1L\gamma}$$
where $\delta_x$ is as always the Dirac measured centered at $x$. Finally, but most crucially, it is a theorem of Masur \cite{Masur} that $\mu_{\Thu}$ is invariant and ergodic under the action of the mapping class group. 

\subsection{Sub-convergence of the measures $\nu_{\gamma_0}^L$}
As we mentioned earlier, we are interested in the behavior of the measures $\nu_{\gamma_0}^L$ defined in \eqref{eq-measure-nu} when $L$ grows. Our first goal is to prove that any accumulation point is a multiple of the Thurston measure $\mu_{\Thu}$:

\begin{prop}\label{prop-sublimit}
Any sequence $(L_n)_n$ of positive numbers with $L_n\to\infty$ has a subsequence $(L_{n_i})_i$ such that the measures $(\nu_{\gamma_0}^{L_{n_i}})_i$ converge in the weak-*-topology to the measure $\alpha\cdot\mu_{\Thu}$ on $\CM\CL(\Sigma)\subset\CC(\Sigma)$ for some $\alpha>0$.
\end{prop}

In preparation to prove Proposition \ref{prop-sublimit} fix $\epsilon$ as in the beginning of section \ref{sec-themap} and consider the corresponding set $\CS^\epsilon_{\gamma_0}$ of those elements in $\CS_{\gamma_0}$ which fill a maximal $\epsilon$-geodesic radalla. Let
$$\pi_{\epsilon,\gamma_0}=\pi_{\epsilon,k}\vert_{\CS_{\gamma_0}^\epsilon}:\CS_{\gamma_0}^\epsilon\to\CM\CL_\BZ(\Sigma)$$
be the restriction of the map \eqref{eq-themap} to $\CS_{\gamma_0}^\epsilon$. 

In order to prove Proposition \ref{prop-sublimit} we will make use of another family of measures on $\CC(\Sigma)$:
\begin{equation}
\mu_{\epsilon,\gamma_0}^L=\frac 1{L^{6g-6+2r}}\sum_{\gamma\in\CM\CL_\BZ}\vert\pi_{\epsilon,\gamma_0}^{-1}(\gamma)\vert\delta_{\frac 1L\gamma}\label{eq-measure-muL}
\end{equation}
The measures $\nu_{\gamma_0}^L$ and $\mu_{\epsilon,\gamma_0}^L$ are closely related to each other, and in fact the first step of the proof of Proposition \ref{prop-sublimit} is to show that the asymptotic behaviors of the measures $\nu_{\gamma_0}^L$ and $\mu_{\epsilon,\gamma_0}^L$ are identical:

\begin{lem}\label{alllimitsthesame}
Let $(L_n)_n$ be a sequence tending to $\infty$. If one of the limits 
$$\lim_{n\to\infty}\mu_{\epsilon,\gamma_0}^{L_n}\text{ and }\lim_{n\to\infty}\nu_{\gamma_0}^{L_n}$$
exists with respect to the weak-*-topology on the space of locally finite measures on the space of currents $\CC(\Sigma)$, then the other also exists and both agree.
\end{lem}

Note that it follows from \eqref{curvesascurrents} and \eqref{curvesascurrents2} that the measures $\nu_{\gamma_0}^L$ and $\mu_{\epsilon,\gamma_0}^L$ are all supported by $\CC_K(\Sigma)$ for some fixed compact set $K\subset PT\Sigma$ -- we will use this fact a number of times in the following pages. 

\begin{proof}
It will be convenient to consider the restriction of the measure $\nu_{\gamma_0}^L$ to the set $\frac 1L\CS^\epsilon_{\gamma_0}$:
\begin{equation}
\nu_{\epsilon,\gamma_0}^L=\frac 1{L^{6g-6+2r}}\sum_{\gamma\in\CS^\epsilon_{\gamma_0}}\delta_{\frac 1L\gamma}\label{eq-measure-nuL}
\end{equation}
Since the set $\CS_{\gamma_0}\setminus\CS_{\gamma_0}^\epsilon$ is negligible by Proposition \ref{generic}, we have that 
$$\lim_{n\to\infty}\frac 1{L_n^{6g-6+2r}}\vert\{\gamma\in\CS_{\gamma_0}\setminus\CS_{\gamma_0}^\epsilon\vert\iota(\lambda_\Sigma,\gamma)\le C\cdot L_n\}\vert=0$$
for all $C>0$. Here $\lambda_\Sigma$ is the Liouville current of the hyperbolic metric on $\Sigma$ and we remind the reader that $\ell_\Sigma(\gamma)=\iota(\lambda_\Sigma,\gamma)$ for every multicurve $\gamma$. By the very definition of the measures $\nu_{\gamma_0}^L$ and $\nu_{\epsilon,\gamma_0}^L$, this means that the difference
$$\nu_{\gamma_0}^{L_n}(\{\lambda\in\CC_K(\Sigma)\vert\iota(\lambda_\Sigma,\lambda)\le C\})-\nu_{\epsilon,\gamma_0}^{L_n}(\{\lambda\in\CC_K(\Sigma)\vert\iota(\lambda_\Sigma,\lambda)\le C\})\to 0$$
of measures tends to $0$ when $n$ grows. Since, varying $C$, the sets $\{\lambda\in\CC_K(\Sigma)\vert\iota(\lambda_\Sigma,\lambda)\le C\}$ form a compact exhaustion of $\CC_K(\Sigma)$ we deduce that whenever one of the limits 
$$\lim_{n\to\infty}\nu_{\epsilon,\gamma_0}^{L_n}\text{ and }\lim_{n\to\infty}\nu_{\gamma_0}^{L_n}$$
exists, then the other also exists and both agree.

It follows that it suffices to prove that $\lim_{n\to\infty}\nu_{\epsilon,\gamma_0}^{L_n}=\mu$ if and only if $\lim_{n\to\infty}\mu_{\epsilon,\gamma_0}^{L_n}=\mu$. Note that $\mu_{\epsilon,\gamma_0}^L$ is the push-forward of the measure $\nu_{\epsilon,\gamma_0}^L$ under the map
$$\pi_{\epsilon,\gamma_0}^L:\frac 1L\CS_{\gamma_0}^\epsilon\to\CM\CL(\Sigma),\ \ \frac 1L\gamma\mapsto\frac 1L\pi_{\epsilon,\gamma_0}(\gamma).$$
In particular, the desired result follows easily once we prove that $\pi^L_{\epsilon,\gamma_0}$, when $L$ is large, almost does not move points. More precisely, fixing for the sake of concreteness a distance $d:\CC_K(\Sigma)\times\CC_K(\Sigma)\to\BR_+$ inducing the topology of $\CC_K(\Sigma)$ (see \cite{Chris-moon-kasra} for a concrete choice), we prove:

\begin{fact}\label{doesnotmove}
For all $C,\rho>0$ there is $L_0$ with 
$$d\left(\frac 1L\gamma,\pi_{\epsilon,\gamma_0}^L\left(\frac 1L\gamma\right)\right)\le\rho$$
for every $\gamma\in\CS^\epsilon_{\gamma_0}(\Sigma)$ with $\iota(\lambda_\Sigma,\gamma)\le C\cdot L$ and for every $L\ge L_0$.
\end{fact}
\begin{proof}
We argue by contradiction: Suppose there exist $C,\rho$ positive, a sequence $L_n\to\infty$, and a sequence $\gamma_n\in\CS^\epsilon_{\gamma_0}$ with $\iota(\lambda_\Sigma,\gamma)\le C\cdot L_n$ satisfying $d\left(\frac 1{L_n}\gamma,\pi_{\epsilon,\gamma_0}^{L_n}\left(\frac 1{L_n}\gamma\right)\right)\ge\rho$. Both sequences $\left(\frac 1{L_n}\gamma_n\right)$ and $\left(\pi_{\epsilon,\gamma_0}^{L_n}\left(\frac 1{L_n}\gamma_n\right)\right)=\left(\frac 1{L_n}\pi_{\epsilon,\gamma_0}\left(\gamma_n\right)\right)$ are contained in the compact set $\{\lambda\in\CC_K(\Sigma)\vert\iota(\lambda_\Sigma,\lambda)\le C\}$. It follows, by passing to a subsequence, that we can assume that they converge: 
$$\frac 1{L_n}\gamma_n\to\mu,\ \ \pi_{\epsilon,\gamma_0}^{L_n}\left(\frac 1{L_n}\gamma_n\right)=\frac 1{L_n}\pi_{\epsilon,\gamma_0}\left(\gamma_n\right)\to\mu'.$$
By construction, $\mu'$ is a limit of simple curves and hence it is a measured lamination. On the other hand, by continuity of the intersection form we have
$$\iota(\mu,\mu)=\lim_{n\to\infty}\iota\left(\frac 1{L_n}\gamma_n,\frac 1{L_n}\gamma_n\right)=\lim_{n\to\infty}\frac{\iota(\gamma_n,\gamma_n)}{L_n^2}=\lim_{n\to\infty}\frac{\iota(\gamma_0,\gamma_0)}{L_n^2}=0,$$
proving that also $\mu\in\CM\CL(\Sigma)$. 

Recall that by Lemma \ref{finite} and the definition of $\CS^\epsilon_{\gamma_0}$, there is a finite set of $\epsilon$-geodesic radallas carrying all curves in $\CS^\epsilon_{\gamma_0}$. In particular, by passing to a further subsequence, we can assume there is a fixed $\epsilon$-geodesic radalla $\hat\tau$ with $\gamma_n\prec\hat\tau$ for all $n$. Let $\tau\subset\hat\tau$ be the underlying train-track, and $\eta_1,\dots,\eta_{s}$ curves transversal to $\tau$ and satisfying \eqref{FNML}. Since $\eta_i$ is transversal to $\hat\tau$ we get from Lemma \ref{transverse-inter} that
$$\iota(\gamma_n,\eta_i)=\iota(\pi_{\epsilon,\gamma_0}(\gamma_n),\eta_i)$$
for all $n$. The continuity of the intersection form implies that
$$\iota(\mu,\eta_i)=\iota(\mu',\eta_i).$$
Since $i$ was arbitrary we have $\mu=\mu'$ by \eqref{FNML}. But this implies that $\frac 1{L_n}\gamma_n$ and $\pi_{\epsilon,\gamma_0}^{L_n}(\frac 1{L_n}\gamma_n)$ have the same limit and hence that their distance tends to $0$, contradicting our assumption. We have proved Fact \ref{doesnotmove}.
\end{proof}

Now, suppose that $f:\CC_K(\Sigma)\to\BR$ is continuous with compact support, and note that there is some $C$ such that the interior of the compact set $Z=\{\lambda\in\CC_K(\Sigma)\vert\iota(\lambda_\Sigma,\lambda)\le C\}$ contains the support of $f$. Note also that $f$ is uniformly continuous, meaning that for all $\epsilon_0$ there is $\rho$ with $\vert f(\alpha)-f(\alpha')\vert\le\epsilon_0$ whenever $d(\alpha,\alpha')\le\rho$. This means that with $L_0$ as in Fact \ref{doesnotmove} and for all $L_n\ge L_0$ and all $\gamma\in\CS_{\gamma_0}^\epsilon$ we have 
$$\left\vert f\left(\frac 1{L_n}\gamma\right)-f\left(\pi_{\epsilon,\gamma_0}^{L_n}\left(\frac 1{L_n}\gamma\right)\right)\right\vert\le\epsilon_0$$
Since $\mu_{\epsilon,\gamma_0}^L$ is the push-forward of the measure $\nu_{\epsilon,\gamma_0}^L$ under $\pi_{\epsilon,\gamma_0}^L$, this implies that
\begin{align*}
&\left\vert\int_{\CC_K(\Sigma)}f(\lambda)d\mu_{\epsilon,\gamma_0}^{L_n}(\lambda)-\int_{\CC_K(\Sigma)}f(\lambda)d\nu_{\epsilon,\gamma_0}^{L_n}(\lambda)\right\vert\le\epsilon_0\nu_{\epsilon,\gamma_0}^{L_n}(Z)\ \ \text{and}\\
&\left\vert\int_{\CC_K(\Sigma)}f(\lambda)d\mu_{\epsilon,\gamma_0}^{L_n}(\lambda)-\int_{\CC_K(\Sigma)}f(\lambda)d\nu_{\epsilon,\gamma_0}^{L_n}(\lambda)\right\vert\le\epsilon_0\mu_{\epsilon,\gamma_0}^{L_n}(Z)
\end{align*}
for any $\epsilon_0$ small enough. If $\lim_{n\to\infty}\nu_{\epsilon,\gamma_0}^{L_n}=\mu$, we get that $\nu_{\epsilon,\gamma_0}^{L_n}(Z)$ is bounded independently of $n$. Since $\epsilon_0$ was arbitrary we get thus that 
$$\lim_{n\to\infty}\left(\int_{\CC_K(\Sigma)}f(\lambda)d\mu_{\epsilon,\gamma_0}^{L_n}(\lambda)-\int_{\CC_K(\Sigma)}f(\lambda)d\nu_{\epsilon,\gamma_0}^{L_n}(\lambda)\right)=0$$
which implies, because $f$ was arbitrary, that also $\lim_{n\to\infty}\mu_{\epsilon,\gamma_0}^{L_n}=\mu$.

The same argument proves also that if $\lim_{n\to\infty}\mu_{\epsilon,\gamma_0}^{L_n}=\mu$ then we also have $\lim_{n\to\infty}\nu_{\epsilon,\gamma_0}^{L_n}=\mu$.
\end{proof}

We are finally ready to prove Proposition \ref{prop-sublimit}:

\begin{proof}[Proof of Proposition \ref{prop-sublimit}]
To begin we prove that the sequence $(\mu_{\epsilon,\gamma_0}^{L_n})_n$ has a convergent subsequence. In order to do so it suffices to prove that for every compact set $K\subset\CM\CL(\Sigma)$ with $\mu_{\Thu}(\D K)=0$ the sequence $(\mu_{\epsilon,\gamma_0}^{L_n}(K))_n$ is bounded. 

Note that
\begin{align*}
\mu_{\epsilon,\gamma_0}^{L_n}(K)
&=\frac 1{L_n^{6g-6+2r}}\sum_{\tiny
\CM\CL_\BZ\cap L_n\cdot K}\vert\pi_{\epsilon,\gamma_0}^{-1}(\lambda)\vert\\
&\le\frac \kappa{L_n^{6g-6+2r}}\vert \CM\CL_\BZ\cap L_n\cdot K\vert
\end{align*}
where $\kappa$ is the constant provided by Proposition \ref{prop-key-map}. Since the last quantity converges to $\kappa\cdot\mu_{\Thu}(K)$ when $n\to\infty$, it follows that our original sequence is bounded, as we needed to prove. 
At this point we know that the sequence of measures $(\mu_{\epsilon,\gamma_0}^{L_{n}})_n$ contains a subsequence $(\mu_{\epsilon,\gamma_0}^{L_{n_i}})_i$ which converges to some measure $\mu$. Moreover, since $\mu(K)\le\kappa\cdot\mu_{\Thu}(K)$ for every $K$ we deduce that the limit $\mu$ is absolutely continuous with respect to the Thurston measure, with Radon-Nikodym derivative $\frac{d\mu}{d\mu_{\Thu}}$ bounded by $\kappa$. 

Note at this point that it follows from Lemma \ref{alllimitsthesame} that we also have
$$\lim_{i\to\infty}\nu_{\gamma_0}^{L_{n_i}}=\mu.$$
This means that it only remains to be proved that the Radon-Nikodym derivative $\frac{d\mu}{d\mu_{\Thu}}$ is essentially constant. To that end note that $\nu_{\gamma_0}^{L_{n_i}}$ is invariant under the mapping class group, since for all $U\subset\CC(\Sigma)$ and all $\phi\in\Map(\Sigma)$ we have
\begin{align*}
\phi_*\nu_{\gamma_0}^{L_{n_i}}(U)
&=\nu_{\gamma_0}^{L_{n_i}}(\phi(U))=\frac 1{L_{n_i}^{6g-6+2r}}\vert\CS_{\gamma_0}\cap\phi({L_{n_i}}\cdot U)\vert\\
&=\frac 1{L_{n_i}^{6g-6+2r}}\vert\phi^{-1}\CS_{\gamma_0}\cap {L_{n_i}}\cdot U\vert=\frac 1{L_{n_i}^{6g-6+2r}}\vert\CS_{\gamma_0}\cap {L_{n_i}}\cdot U\vert\\
&=\nu_{\gamma_0}^{L_{n_i}}(U).
\end{align*}
This implies that the measure $\mu$ is also invariant under the mapping class group. Hence, the Radon-Nikodym derivative $\frac{d\mu}{d\mu_{\Thu}}$ is a mapping class group invariant measurable function on $\CM\CL(\Sigma)$. Since the Thurston measure $\mu_{\Thu}$ is ergodic with respect to the mapping class group action \cite{Masur}, it follows that $\frac{d\mu}{d\mu_{\Thu}}$ is essentially constant as we needed to prove.
\end{proof}

\subsection{Limits of measures and counting}\label{subsec-counting}
After the hard work of proving Proposition \ref{prop-sublimit} we can reap some of its consequences. They are all based on the following simple observation:

\begin{prop}\label{prop-reduction}
Let $(L_n)_n$ be a sequence with $\lim_{n\to\infty}\nu_{\gamma_0}^{L_n}=\alpha\cdot\mu_{\Thu}$ for some $\alpha\in\BR_+$. Then
$$\lim_{n\to\infty}\frac{\vert\{\gamma\in\CS_{\gamma_0}\vert\iota(\lambda_0,\gamma)\le L_n\}\vert}{{L_n}^{6g-6+2r}}=\alpha\cdot\mu_{\Thu}(\{\lambda\in\CM\CL(\Sigma)\vert\iota(\lambda_0,\lambda)\le 1\})$$
for every filling current $\lambda_0\in\CC(\Sigma)$.
\end{prop}

\begin{proof}
We have that
\begin{align*}
\lim_{n\to\infty}\frac{\vert\{\gamma\in\CS_{\gamma_0}\vert\iota(\lambda_0,\gamma)\le L_n\}\vert}{{L_n}^{6g-6+2r}}
&=\lim_{n\to\infty}\nu_{\gamma_0}^{L_n}(\{\lambda\in\CC_K(\Sigma)\vert\iota(\lambda_0,\lambda)\le1\})\\
&=\alpha\cdot\mu_{\Thu}(\{\lambda\in\CC_K(\Sigma)\vert\iota(\lambda_0,\lambda)\le1\})\\
&=\alpha\cdot\mu_{\Thu}(\{\lambda\in\CM\CL(\Sigma)\vert\iota(\lambda_0,\lambda)\le1\}).
\end{align*}
Here the first equality follows from the very definition of the measures $\nu_{\gamma_0}^L$. The second equality follows from the assumption in the proposition because as we noted earlier $\mu_{\Thu}(\{\lambda\in\CC_K(\Sigma)\vert\iota(\lambda_0,\lambda)=1\})=0$. Finally, the last equality holds because $\mu_{\Thu}$ is supported by $\CM\CL(\Sigma)$.
\end{proof}

We will later prove that the limit $\lim_{L\to\infty}\nu_{\gamma_0}^L$ actually exists if $\Sigma$ is a once-punctured torus, which in light of Proposition \ref{prop-reduction} will mean that the corresponding limit $\lim_{L\to\infty}\frac{\vert\{\gamma\in\CS_{\gamma_0}\vert\iota(\lambda_0,\gamma)\le L\}\vert}{{L}^{2}}$ also exists. For a general surface we only prove that the following weaker statement:

\begin{kor}\label{kor-ratio}
Let $\Sigma$ be a hyperbolic surface of finite area, and let $\lambda_1,\lambda_2\in\CC(\Sigma)$ be filling currents. Then we have
$$\lim_{L\to\infty}\frac{\vert\{\gamma\in\CS_{\gamma_0}\vert\iota(\lambda_1,\gamma)\le L\}\vert}{\vert\{\gamma\in\CS_{\gamma_0}\vert\iota(\lambda_2,\gamma)\le L\}\vert}=\frac{\mu_{\Thu}(\{\lambda\in\CM\CL(\Sigma)\vert\iota(\lambda_1,\lambda)\le 1\})}{\mu_{\Thu}(\{\lambda\in\CM\CL(\Sigma)\vert\iota(\lambda_2,\lambda)\le 1\})}$$
for every multicurve $\gamma_0$ in $\Sigma$. Here $\mu_{\Thu}$ is as always the Thurston measure on the space of measured laminations $\CM\CL(\Sigma)$.
\end{kor}
\begin{proof}
It suffices to prove that every sequence $(L_n)$ has a subsequence $(L_{n_i})$ for which the claimed equality holds. Well, from Proposition \ref{prop-sublimit} we know that $(L_n)$ has a subsequence for which $\nu_{\gamma_0}^{L_{n_i}}$ converges to $\alpha\cdot\mu_{\Thu}$ for some $\alpha>0$. Thus by Proposition \ref{prop-reduction} we get that
$$\lim_{i\to\infty}\frac {\vert\{\gamma\in\CS_{\gamma_0}\vert\iota(\lambda_1,\gamma)\le L_{n_i}\}\vert}{{L_{n_i}}^{6g-6+2r}}=\alpha\cdot\mu_{\Thu}(\{\lambda\in\CM\CL(\Sigma)\vert\iota(\lambda_1,\lambda)\le 1\})$$
and that
$$\lim_{i\to\infty}\frac {\vert\{\gamma\in\CS_{\gamma_0}\vert\iota(\lambda_2,\gamma)\le L_{n_i}\}\vert}{{L_{n_i}}^{6g-6+2r}}=\alpha\cdot\mu_{\Thu}(\{\lambda\in\CM\CL(\Sigma)\vert\iota(\lambda_2,\lambda)\le 1\})$$
The claim follows by taking the quotient of these two equations.
\end{proof}

As we mentioned in the introduction, combining Corollary \ref{kor-ratio} and the work of Mirzakhani \cite{Maryam-new} one gets immediately that the limit 
$$\lim_{L\to\infty}\frac{\vert\{\gamma\in\CS_{\gamma_0}\vert\iota(\lambda_0,\gamma)\le L\}\vert}{L^{6g-6+2r}}$$
exists for every multicurve $\gamma_0$ and every filling current $\lambda_0$. It is thus natural to wonder if more generally the limit
\begin{equation}\label{eq-1001}
\lim_{L\to\infty}\frac{\vert\{f(\alpha)\vert f\in\Map(\Sigma),\ \iota(\lambda_0,f(\alpha))\le L\}\vert}{L^{6g-6+2r}}
\end{equation}
also exists for currents $\alpha\in\CC(\Sigma)$ other than those arising from multicurves. This is indeed the case. To explain why, suppose for the sake of concreteness that $\Sigma$ is closed of genus $g$, that $\alpha\in\CC(\Sigma)$ has trivial stabilizer in the mapping class group, and consider for $L>0$ the measure
$$\nu_{\alpha}^L=\frac 1{L^{6g-6}}\sum_{f\in\Map(\Sigma)}\delta_{\frac 1Lf(\alpha)}$$
on the space of currents. Noting that there is some $K\ge 1$ with 
$$\frac 1K\cdot\iota(\lambda,\lambda_\Sigma)\le\iota(\lambda,\alpha)\le K\cdot\iota(\lambda,\lambda_\Sigma)$$
for all $\lambda\in\CC(\Sigma)$, we get that the $\nu_{\alpha}^L$-measure of the compact sets $\{\lambda\in\CC(\Sigma)\vert\iota(\lambda,\lambda_\Sigma)\le C\}$ is bounded from above for all $C$ by some number which does not depend on $L$. It follows that every sequence $(\nu_{\alpha}^{L_n})_n$ has a convergent subsequence. Moreover, as in the proof of Proposition \ref{prop-reduction} we have that if $\nu_{\alpha}^{L_n}$ converges to $\mu$ then 
$$\lim_{n\to\infty}\frac{\vert\{f\in\Map(\Sigma)\vert \iota(\lambda_0,f(\alpha))\le L_n\}\vert}{L_n^{6g-6}}
=\mu(\{\lambda\in\CC(\Sigma)\vert\iota(\lambda_0,\lambda)\le 1\})$$
In particular, to prove that the limit \eqref{eq-1001} exists it suffices to prove that the quantity 
\begin{equation}\label{eq-1002}
\mu(\{\lambda\in\CC(\Sigma)\vert\iota(\lambda_0,\lambda)\le 1\})
\end{equation}
does not depend on the particular sequence $L_n$. Well, we want to prove the independence of \eqref{eq-1002} for every every filling current $\lambda_0$, but since weighted currents corresponding to multicurves are dense in $\CC(\Sigma)$, it suffices to prove it for multicurves $\lambda_0$, which we moreover might assume to be filling and with trivial stabilizer in the mapping class group. In this case we have that
\begin{align*}
\mu(\{\lambda\in\CC(\Sigma)\vert\iota(\lambda_0,\lambda)\le 1\})
   &=\lim_{n\to\infty}\frac{\vert\{f\in\Map(\Sigma)\vert \iota(\lambda_0,f(\alpha))\le L_n\}\vert}{L_n^{6g-6}}\\
   &=\lim_{n\to\infty}\frac{\vert\{f\in\Map(\Sigma)\vert \iota(\alpha,f(\lambda_0))\le L_n\}\vert}{L_n^{6g-6}}\\
   &=\lim_{n\to\infty}\frac{\vert\{\gamma\in\CS_{\lambda_0}\vert \iota(\alpha,\gamma)\le L_n\}\vert}{L_n^{6g-6}}
\end{align*}
where the second equality holds because the intersection form is symmetric and invariant under the mapping class group -- the third limit is just a rewriting of the second one. However, from the work of Mirzakhani \cite{Maryam-new} and Corollary \ref{kor-ratio} we get that the third limit does not depend on the sequence $L_n$. It follows thus that the limit \eqref{eq-1001} exists, as claimed.

\subsection{Criteria for convergence}
We discuss now some conditions ensuring that the measures $\nu_{\gamma_0}^L$ converge when $L\to\infty$. Note that by Lemma \ref{alllimitsthesame} it suffices to give conditions ensuring the convergence of the measures $\mu_{\epsilon,\gamma_0}^L$ defined in \eqref{eq-measure-muL}. The first basic observation is the following:

\begin{lem}\label{lem-criterion-convergence-measures-sick-of-this}
Suppose that there is a non-empty open set $U\subset\CM\CL(\Sigma)$ with $\mu_{\Thu}(\D U)=0$ and for which the limit $\lim_{L\to\infty}\mu_{\epsilon,\gamma_0}^L(U)=\alpha_U$ exists and is finite. Then the limit $\lim_{L\to\infty}\nu_{\gamma_0}^L=\mu$ also exists and we have $\mu(U)=\alpha_U$.
\end{lem}
\begin{proof}
Recall that by Proposition \ref{prop-sublimit} every sequence $L_n\to\infty$ contains a subsequence $(L_{n_i})_i$ such that $(\nu_{\gamma_0}^{L_{n_i}})_i$ converges in the weak-*-topology to a multiple $\alpha\cdot\mu_{\Thu}$ of the Thurston measure. By Lemma \ref{alllimitsthesame} we get that 
$$\lim_{i\to\infty}\mu_{\epsilon,\gamma_0}^{L_{n_i}}=\alpha\cdot\mu_{\Thu}$$
as well. Therefore 
$$\alpha\cdot\mu_{\Thu}(U)=\lim_{i\to\infty}\mu_{\epsilon,\gamma_0}^{L_{n_i}}(U)=\alpha_U$$ 
meaning that the constant $\alpha$ does not depend on the particular convergent sequence $(\nu_{\gamma_0}^{L_{n_i}})_i$. The claim follows.
\end{proof}

Next we give a criterion ensuring the existence of a limit for the measures $\nu_{\gamma_0}^L$ in terms of the existence of densities of sets of simple multicurves invariant under a certain semi-group. More precisely, if $\tau$ is train-track let 
\begin{equation}\label{eq-semigroup-tt}
\Gamma_\tau=\{\phi\in\Map(\Sigma)\vert\phi(\tau)\prec\tau\}
\end{equation}
be the semi-group consisting of those mapping classes which map $\tau$ to a train-track carried by $\tau$. 

\begin{prop}\label{prop-semigroup-limit}
Let $\tau$ be a maximal recurrent train-track and $U\subset\{\lambda\in\CM\CL(\Sigma)\vert\lambda\prec\tau\}$ open with $\mu_{\Thu}(U)>0$ and $\mu_{\Thu}(\D U)=0$. Suppose also that the following holds:
\begin{quote}
(*) If $\CI\subset\{\gamma\in\CM\CL_\BZ(\Sigma)\vert\gamma\prec\tau\}$ is a non-empty $\Gamma_\tau$-invariant set of simple multicurves carried by $\tau$ then there is $\alpha>0$ with
$$\lim_{L\to\infty}\frac 1{L^{6g-6+2r}}\vert\CI\cap L\cdot U\vert=\alpha\cdot\mu_{\Thu}(U).$$
\end{quote}
Then the limit $\lim_{L\to\infty}\nu_{\gamma_0}^L$ exists.
\end{prop}

Before proving Proposition \ref{prop-semigroup-limit} let us comment on two of its features. On the one hand, Proposition \ref{prop-semigroup-limit} has the virtue that it reduces the problem of showing that the measures $\nu_{\gamma_0}^L$ converge to a problem about distribution of simple multicurves. On the other hand, working with semigroups is harder than working with groups. Also, even if we were to replace the semigroup by the whole mapping class group, the statement would still not be at all obvious -- in fact, it would be the main result of \cite{Maryam}. 

\begin{proof}
Note first that, up to choosing a different hyperbolic metric on the surface $\Sigma$, we can assume $\tau$ is $\epsilon$-geodesic. In particular, we get from Proposition \ref{prop-key-map} that if $\phi\in\Gamma_\tau$ then
$$\vert\pi_{\epsilon,\gamma_0}^{-1}(\gamma)\vert\le\vert\pi_{\epsilon,\gamma_0}^{-1}(\phi(\gamma))\vert$$
for every simple multicurve $\gamma\fills\tau$. It follows that the set 
$$\CI_s=\{\gamma\in\CM\CL_\BZ(\Sigma),\ \gamma\fills\tau,\ \vert\pi_{\epsilon,\gamma_0}^{-1}(\gamma)\vert\ge s\}$$
is $\Gamma_\tau$-invariant for all $s$. Note also that Proposition \ref{prop-key-map} asserts that there is $\kappa$ with $\CI_s=\emptyset$ for all $s>\kappa$. Now, by assumption, for all $s$ there is some number $\alpha_s$ with 
$$\lim_{L\to\infty}\frac 1{L^{6g-6+2r}}\vert\CI_s\cap L\cdot U\vert=\alpha_s\cdot\mu_{\Thu}(U).$$
Also, we can compute $\mu_{\epsilon,\gamma_0}^L(U)$ as follows:
\begin{align*}
\mu_{\epsilon,\gamma_0}^L(U)
   &=\frac 1{L^{6g-6+2r}}\sum_{\gamma\in\CM\CL_\BZ(\Sigma)\cap L\cdot U}\vert\pi_{\epsilon,\gamma_0}^{-1}(\gamma)\vert\\
   &\sim\frac 1{L^{6g-6+2r}}\sum_{\gamma\in\CM\CL_\BZ(\Sigma)\cap L\cdot U,\ \gamma\fills\tau}\vert\pi_{\epsilon,\gamma_0}^{-1}(\gamma)\vert\\
   &=\frac 1{L^{6g-6+2r}}\sum_{s=1}^k\vert\CI_s\cap L\cdot U\vert.
\end{align*}
Here $\sim$ means that the difference between the quantities tends to $0$ when $L$ grows -- this is so because the set of multicurves carried by $\tau$ but which do not fill $\tau$ is negligible. The final equality holds because $\vert\pi_{\epsilon,\gamma_0}^{-1}(\gamma)\vert$ can only take the values $0,1,\dots,\kappa$.

Combining the previous equations we get that $\mu_{\epsilon,\gamma_0}^L(U)$ converges when $L\to\infty$ to the number 
$$\lim_{L\to\infty}\mu_{\epsilon,\gamma_0}^L(U)=(\alpha_1+\alpha_2+\dots+\alpha_\kappa)\cdot\mu_{\Thu}(U).$$
The claim now follows from Lemma \ref{lem-criterion-convergence-measures-sick-of-this}.
\end{proof}

At this point we can prove:

\begin{kor}\label{kor-small-k}
Let $\Sigma$ be a hyperbolic surface of genus $g$ with $r$ cusps and suppose that $2g+r>3$. Then we have:
\begin{align*}
\lim_{L\to\infty}\frac{\vert\{\gamma\in\CS\vert\ell_\Sigma(\gamma)\le L,\ \iota(\gamma,\gamma)=1\}\vert}{L^{6g-6+2r}}
&=3(2g-2+r)\cdot c_\Sigma\\
\lim_{L\to\infty}\frac{\vert\{\gamma\in\CS\vert\ell_\Sigma(\gamma)\le L,\ \iota(\gamma,\gamma)=2\}\vert}{L^{6g-6+2r}}
&=\frac 92\big((2g+r)(2g+r-3)+2\big)\cdot c_\Sigma.
\end{align*}
where $c_\Sigma=\mu_{\Thu}(\{\lambda\in\CM\CL(\Sigma)\vert\ell_\Sigma(\lambda)\le 1\})$.
\end{kor}

Before launching the proof of Corollary \ref{kor-small-k} note that all the results proved in this section for the measures $\nu_{\gamma_0}^L$ and $\mu_{\epsilon,\gamma_0}^L$ also apply to the measures
$$\nu_k^L=\frac 1{L^{6g-6+2r}}\sum_{\gamma\in\CS_k}\delta_{\frac 1L\gamma}\text{ and } \mu_{\epsilon,k}^L=\frac 1{L^{6g-6+2r}}\sum_{\gamma\in\CM\CL_\BZ}\vert\pi_{\epsilon,k}^{-1}(\gamma)\vert\delta_{\frac 1L\gamma}.$$
The proofs are identical.

\begin{proof}
Let $\tau\subset\Sigma$ be a maximal recurrent train-track and $O\subset\CM\CL(\Sigma)$ the set of measured laminations carried by $\tau$ and let $U=\{\lambda\in O\vert\ell_\Sigma(\lambda)\le 1\}$. Finally, let $\CV\subset\CM\CL_\BZ(\Sigma)\cap O$ be the set of all simple multicurves $\gamma\prec\tau$ carried by $\tau$ and traversing each edge of $\tau$ at least 10 times, that is $\omega_\gamma(e)\ge 10$ for all $e\in E(\tau)$. Note that $\CV$ is generic in the set $\CM\CL_\BZ(\Sigma)\cap O$ of simple multicurves carried by $\tau$. Genericity of $\CV$ in $\CM\CL_\BZ(\Sigma)\cap O$, together with the universal bound given by Proposition \ref{prop-key-map} on the cardinality of the fibers of $\pi_{\epsilon,k}$, implies that 
$$\lim_{L\to\infty}\frac 1{L^{6g-6+2r}}\sum_{\gamma\in(\CM\CL_\BZ(\Sigma)\cap L\cdot U)\setminus\CV}\vert\pi_{\epsilon,k}^{-1}(\gamma)\vert=0$$
This means that if either one of the limits
$$\lim_{L\to\infty}\mu_{\epsilon,k}^L(U)\text{ and }\lim_{L\to\infty}\frac 1{L^{6g-6+2r}}\sum_{\gamma\in\CV\cap L\cdot U}\vert\pi_{\epsilon,k}^{-1}(\gamma)\vert$$ 
exists then the other also exists, and when the limits exist, they agree. 

Consider the cases $k=1$ and recall that by Lemma \ref{countk12} we have 
$$(\pi^1_\epsilon)^{-1}(\gamma)=3(2g-2+r)$$
for all $\gamma\in\CV$. We thus have
$$\lim_{L\to\infty}\frac 1{L^{6g-6+2r}}\sum_{\gamma\in\CV\cap L\cdot U}\vert\pi_{\epsilon,1}^{-1}(\gamma)\vert=3(2g-2+r)
\lim_{L\to\infty}\frac{\vert \CV\cap L\cdot U\vert}{L^{6g-6+2r}}$$
Using again that $\CV$ is generic in $U\cap\CM\CL_\BZ(\Sigma)$ we get that the latter limit converges to the Thurston measure of $U$:
$$\mu_{\Thu}(U)=\lim_{L\to\infty}\frac{\vert \CV\cap L\cdot U\vert}{L^{6g-6+2r}}$$
Altogether we have proved that 
$$\lim_{L\to\infty}\mu_{\epsilon,1}^L(U)=3(2g-2+r)\cdot \mu_{\Thu}(U)$$
Lemma \ref{lem-criterion-convergence-measures-sick-of-this} implies that that the measures $\mu_{\epsilon,1}^L$ converge to a measure $\mu$ with $\mu(U)=3(2g-2+r)\cdot \mu_{\Thu}(U)$. Lemma \ref{alllimitsthesame} implies that $\mu=\lim_{L\to\infty}\nu_1^L$ and we get from Proposition \ref{prop-sublimit} that $\mu$ is a multiple of the Thurston measure. Since we know the measure of $U$ we thus get that
$$\mu=3(2g-2+r)\cdot\mu_{\Thu}$$
The claim for $k=1$ follows now from Proposition \ref{prop-reduction}. Exactly the same argument applies for $k=2$.
\end{proof}


\section{}\label{sec-torus}

In this section we prove Theorem \ref{sat2}. We start by recalling a few facts about curves and train-tracks in the once punctured torus.

\subsection{Train-tracks and simple curves in the punctured torus}
Let $T=T_{1,1}$ be the once punctured torus and recall the following well-known fact:

\begin{fact}
The inclusion map $T_{1,1}\to T_{1,0}$ from the once punctured torus $T_{1,1}$ to the closed torus $T_{1,0}$ induces a bijection between their respective sets of oriented simple non-peripheral curves. Moreover, this bijection preserves both the geometric and algebraic intersection number.
\end{fact}

It follows that choosing a basis for the homology $H_1(T_{1,1},\BZ)$ of the torus, we can identify the set of simple multicurves in the once punctured torus with the integral homology classes in the torus up to sign:
\begin{equation}\label{eq-adeleiscated}
\CM\CL_\BZ(T_{1,1})=\BZ^2/\pm 1
\end{equation}
In fact, this identification extends to an identification between $\BR^2/\pm 1$ and the space of measured laminations on $T_{1,1}$:
\begin{equation}\label{eq-adeleiscated3}
\CM\CL(T_{1,1})=\BR^2/\pm 1
\end{equation}
This last identification is compatible with the scaling by positive scalars in $\BR_+$ on the left and the right. Moreover, \eqref{eq-adeleiscated} and \eqref{eq-adeleiscated3} are also compatible under the identification
\begin{equation}\label{eq-adeleiscated2}
\Map(T_{1,1})=\SL_2\BZ
\end{equation}
between the mapping class group and $\SL_2\BZ$, where the actions are on the left by mapping classes and on the right by matrix multiplication. 
\medskip

We describe next the maximal recurrent train-tracks in $T$. For every pair of oriented simple essential curves $\alpha,\beta\subset T$ which intersect once, there is a train-track $\tau_{\alpha,\beta}$ which carries the curves $\alpha,\beta$ as well as the simple curve $\alpha\beta$ whose homology class is the sum of the classes of $\alpha$ and $\beta$ (compare with figure \ref{fig1}). Observe that 
$$\tau_{\alpha,\beta}=\tau_{-\alpha,-\beta}=\tau_{\beta,\alpha}=\tau_{-\beta,-\alpha}\text{ but }\tau_{\alpha,\beta}\neq\tau_{\alpha,-\beta}.$$
Moreover, with notation as in \eqref{eq-adeleiscated3}, we have 
$$\{\lambda\in\CM\CL(T_{1,1})\vert\lambda\prec\tau_{\alpha,\beta}\}=\{x\alpha+y\beta\vert x,y\ge0\}/\pm 1$$

\begin{figure}[h]
\includegraphics[width=10cm, height=3cm]{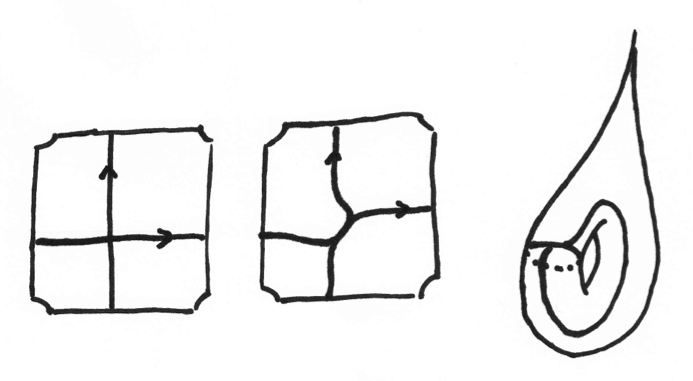}
\caption{The train-track associated to two essential simple curves which intersect once.}
\label{fig1}
\end{figure}

Note that $\tau_{\alpha,\beta}$ is a maximal recurrent train-track. Conversely, it is well-known that every maximal recurrent train-track $\tau$ on the torus $T_{1,1}$ differs from $\tau_{\alpha,\beta}$ by a diffeomorphism. Moreover, since every orientation preserving diffeomorphism $\phi:T_{1,1}\to T_{1,1}$ with $\phi(\tau_{\alpha,\beta})=\tau_{\alpha,\beta}$ is either isotopic to $\Id$ or to $-\Id$ we obtain that the mapping class of the homeomorphism mapping  $\tau$ to $\tau_{\alpha,\beta}$ is unique up to composition by $-\Id$. In other words we have:

\begin{fact}
Let $T$ be a once punctured torus and let $\alpha$ and $\beta$ be oriented simple curves intersecting once. If $\tau$ is a maximal recurrent train-track in $T$ then there is an orientation preserving homeomorphism $\phi:T\to T$ with $\phi(\tau_{\alpha,\beta})=\tau$. Moreover, the mapping class $[\phi]\in\Map(T)$ of $f$ is uniquely determined up to composition by $-\Id$.
\end{fact}

Observe that the train-tracks $\tau_{\alpha\beta,\beta}$ and $\tau_{\alpha,\alpha\beta}$ are carried by $\tau_{\alpha,\beta}$:
$$\tau_{\alpha\beta,\beta},\tau_{\alpha,\alpha\beta}\prec\tau_{\alpha,\beta}.$$
Also, note that if $(\alpha,\beta)$ corresponds to a positively oriented basis of $\BZ^2$ under \eqref{eq-adeleiscated}, we have
$$\delta_{\beta}^{-1}(\tau_{\alpha,\beta})=\tau_{\alpha\beta,\beta},\ \ \delta_{\alpha}(\tau_{\alpha,\beta})=\tau_{\alpha,\alpha\beta}$$
where $\delta_{\beta}$ and $\delta_{\alpha}$ are the right Dehn-twist along $\beta$ and $\alpha$, respectively. We obtain:

\begin{fact}
Let $T$ be a once punctured torus and let $\alpha$ and $\beta$ be oriented simple curves with algebraic intersection number $\langle\alpha,\beta\rangle=1$ and $\phi\in\Map(T)$. Then we have 
$$\Gamma_{\tau_{\alpha,\beta}}=\{\phi\in\Map(T_{1,1})\vert\phi(\tau_{\alpha,\beta})\prec\tau_{\alpha,\beta}\}=\langle\delta_\alpha,\delta_\beta^{-1}\rangle_+$$
where $\langle\delta_\alpha,\delta_\beta^{-1}\rangle_+$ is the semigroup generated by $\delta_\alpha$ and $\delta_\beta^{-1}$.
\end{fact}

Still with the same notation, consider the identification $\pi_1(T_{1,0})=\BZ^2$ with respect to which the two simple curves $\alpha$ and $\beta$ correspond to the standard basis:
$$\alpha=\left(\begin{array}{c}1\\ 0\end{array}\right),\ \ \beta=\left(\begin{array}{c}0\\ 1\end{array}\right).$$
With respect to the induced identification \eqref{eq-adeleiscated2} between the mapping class group $\Map(T)$ and the group $\SL_2\BZ$ we have that
$$\delta_\alpha=\left(\begin{array}{cc}1&1\\ 0&1\end{array}\right),\ \ \delta_\beta=\left(\begin{array}{cc}1&0\\ -1&1\end{array}\right)$$
which means that the semigroup generated by $\delta_\alpha,\delta_\beta^{-1}$ corresponds to the positive semigroup $\SL_2\BN\subset\SL_2\BZ$:
$$\Gamma_{\tau_{\alpha,\beta}}=\langle\delta_\alpha,\delta_\beta^{-1}\rangle_+=\SL_2\BN=\left\{
\left(
\begin{array}{cc}
 a & b \\
c  & d   
\end{array}
\right)\in\SL_2\BZ\middle\vert a,b,c,d\ge 0\right\}$$

\subsection{Radallas and the map $\pi$ in the punctured torus}
Let $\tau=\tau_{\alpha,\beta}$ be a maximal train-track in the punctured torus $T_{1,1}$. The train-track $\tau$ has 3 edges $a,b,c$ labeled in such a way that 
$$\omega_\alpha=\left(
\begin{array}{c} 1 \\ 0 \\ 1
\end{array}
\right)\text{ and }\omega_\beta=\left(
\begin{array}{c} 0 \\ 1 \\ 1
\end{array}
\right)\text{ where }\omega_\gamma=\left(
\begin{array}{c} \omega_\gamma(a) \\ \omega_\gamma(b) \\ \omega_\gamma(c)
\end{array}
\right).$$
Note that the complement $T_{1,1}\setminus\tau$ of this train-track is a punctured bigon, but, after labelling the edges, has a structure reminiscent of a punctured hexagon where opposite sides are identified in $T_{1,1}$. This means that $\tau$ can be drawn as in figure \ref{fig1} or as the boundary of an hexagon in the fundamental domain of an hexagonal torus. See the left part of figure \ref{hand-fig5a}.

\begin{bem}
The only reason why we assumed that $\Sigma$ was not a once punctured torus in earlier sections was because we used the fact that the complementary regions of a maximal train-track could only be triangles and punctured monogons. This additional ``complication" in the case of the torus is in fact not such a problem since we are instead in a very concrete setting. 
\end{bem}

\begin{figure}[h]
\includegraphics[width=8cm, height=3cm]{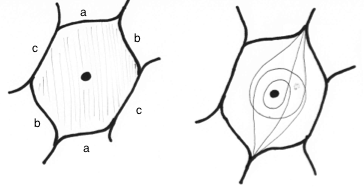}
\caption{The left half is the train-track $\tau$ as the boundary of a hexagon -- the picture is drawn in the universal homology cover of $T_{1,1}$, i.e. in $\BR^2\setminus\BZ^2$. The right image represents a possible radalla extending $\tau$.}
\label{hand-fig5a}
\end{figure}

Suppose that $\hat\tau$ is an arbitrary radalla extending $\tau$. If $\gamma\prec\hat\tau$ then the coefficients $\omega_\gamma(a)$, $\omega_\gamma(b)$ and $\omega_\gamma(c)$ satisfy
$$\omega_\gamma(c)\ge\omega_\gamma(a)+\omega_\gamma(b)$$
with equality if and only if $\gamma\prec\tau$. Let
$$\rho(\gamma)=\omega_\gamma(c)-(\omega_\gamma(a)+\omega_\gamma(b))$$
denote the defect and define a map
$$\pi_{\hat\tau}:\BR^{E(\hat\tau)}\to\BR^{E(\tau)}$$
such that
$$\pi_{\hat\tau}(\omega_\gamma)=\left(
\begin{array}{c} \omega_\gamma(a)+\rho(\gamma) \\ \omega_\gamma(b)+\rho(\gamma) \\ \omega_\gamma(c)+\rho(\gamma).
\end{array}
\right)$$
As in section \ref{sec-themap}, we can give a direct description of the map $\pi_{\hat\tau}$ in terms of the radalla, i.e. without giving formulas for the weights. To do so, we just consider the image $\phi(\hat\tau)$ of the radalla in question and remove all punctured monogons and all bigon and homotope what is left into $\tau$ (cf. with figure \ref{hand-fig5b}). 
\begin{figure}[h]
\includegraphics[width=8cm, height=3cm]{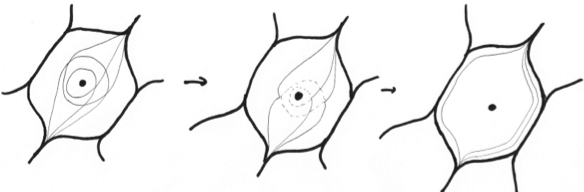}
\caption{The map $\pi_{\hat\tau}$ at the level of radallas.}
\label{hand-fig5b}
\end{figure}
Similarly, as long as we restrict to curves which are carried by some $\epsilon$-geodesic radalla, for $\epsilon$ small enough, we can describe the map without making use of any concrete carrying radalla. Instead we can define the map only using the geodesic representative of the curve in question, just as we did in  \ref{sec-themap}. Hence, as we did in the previous section, we get a well-defined map
$$\pi_\epsilon:\CS_{T_{1,1},k}^\epsilon\to\CM\CL_\BZ(T_{1,1})$$
which satisfies Lemma \ref{transverse-inter} and Proposition \ref{prop-key-map}. Moreover, the arguments in section \ref{sec-measures} remain valid -- in particular Proposition \ref{prop-reduction} and Proposition \ref{prop-semigroup-limit} still hold. We leave the details to the reader.

\subsection{Proof of Theorem \ref{sat2}}
Recall Theorem \ref{thm-density} stated in the introduction:

\begin{named}{Theorem \ref{thm-density}}
Every $\SL_2\BN$-invariant set $\CI\subset\BN^2$ has a density, meaning that there is $\alpha\in\BR$ with  
$$\lim_{L\to\infty}\frac 1{L^2}\vert \CI\cap L\cdot U\vert=\alpha\cdot\vol(U)$$
for any $U\subset\BR^2$ open and bounded by a rectifiable Jordan curve. Here $L\cdot U=\{v\in\BR^2\vert \frac 1Lv\in U\}$ is the set obtained by scaling $U$ by $L$ and $\vol(U)$ is the area of $U$ with respect to Lebesgue measure. 
\end{named}

Assuming Theorem \ref{thm-density} for now, we prove Theorem \ref{sat2}:

\begin{named}{Theorem \ref{sat2}}
Let $\Sigma$ be a complete hyperbolic surface of finite volume homeomorphic to a once punctured torus and let $\gamma_0\subset\Sigma$ be a multicurve. The limit \eqref{eq1} exists and moreover we have
$$\lim_{L\to\infty}\frac{\vert\{\gamma\in\CS_{\gamma_0}\vert\ell_\Sigma(\gamma)\le L\}\vert}{L^2}=C_{\gamma_0}\cdot\mu_{\Thu}(\{\lambda\in\CM\CL(\Sigma)\vert\ell_\Sigma(\lambda)\le 1\})$$
where $\mu_{\Thu}$ is the Thurston measure on the space of measured laminations $\CM\CL(\Sigma)$ and $C_{\gamma_0}>0$ depends only on $\gamma_0$.
\end{named}

\begin{proof}
We start by proving that the measures $\nu_{\gamma_0}^L$ defined in \eqref{eq-measure-muL} converge. By Proposition \ref{prop-semigroup-limit} it suffices to exhibit a maximal recurrent train-track $\tau$ and an open set $U\subset\{\lambda\in\CM\CL(\Sigma)\vert\lambda\prec\tau\}$ with $\mu_{\Thu}(U)>0$ and $\mu_{\Thu}(\D U)=0$ and such that the following holds:
\begin{quote}
(*) If $\CI\subset\{\gamma\in\CM\CL_\BZ(\Sigma)\vert\gamma\prec\tau\}$ is a non-empty $\Gamma_\tau$-invariant set of simple multicurves carried by $\tau$ then there is $\alpha>0$ with
$$\lim_{L\to\infty}\frac 1{L^2}\vert\CI\cap L\cdot U\vert=\alpha\cdot\mu_{\Thu}(U),$$
where $\Gamma_\tau=\{\phi\in\Map(\Sigma)\vert\phi(\tau)\prec\tau\}$.
\end{quote}
Let $\tau=\tau_{\alpha,\beta}$ be a standard maximal train-track in the once punctured torus.  As discussed earlier, identify the set of all simple multicurves carried by $\tau$ with $\BN^2$ and the semi-group $\Gamma_\tau$ with $\SL_2\BN$ in such a way that the action of $\phi\in\Gamma_\gamma$ on curves carried by $\tau$ corresponds to the action by matrix multiplication. In particular, we can identify the $\Gamma_\tau$-invariant set $\CI$ with an $\SL_2\BN$-invariant set $\CI\subset\BN^2$. Moreover, the identification $\{\gamma\in\CM\CL_\BZ\ \vert\ \gamma\prec\tau\}=\BN^2$ extends to an identification
$$\{\lambda\in\CM\CL(T_{1,1})\ \vert\ \lambda\prec\tau\}=\BR_+^2$$
in such a way that scaling by positive reals is preserved. It follows that for 
$$U\subset\{\lambda\in\CM\CL(T_{1,1})\ \vert\ \lambda\prec\tau\}=\BR_+^2$$
open one has 
$$\frac 1{L^2}\vert\CI\cap L\cdot U\vert=\frac 1{L^2}\vert\CI\cap L\cdot U\vert$$
where the left is computed in $\CM\CL(T_{1,1})$ and the right is computed in $\BR^2$. From Theorem \ref{thm-density} we get that, when working in $\BR^2$, the limit $\lim_{L\to\infty}\frac 1{L^2}\vert\CI\cap L\cdot U\vert$ exists and hence (*) holds. We have proved that the limit
$$\nu_{\gamma_0}=\lim_{L\to\infty}\nu_{\gamma_0}^L$$
exists. Moreover, it follows from Proposition \ref{prop-sublimit} that
$$\nu_{\gamma_0}=C_{\gamma_0}\cdot\mu_{\Thu}$$
for some $C_{\gamma_0}\in\BR_+$. Applying Proposition \ref{prop-reduction} to the Liouville current $\lambda_\Sigma$ we get that 
$$\lim_{L\to\infty}\frac{\vert\{\gamma\in\CS_{\gamma_0}\vert\iota(\lambda_\Sigma,\gamma)\le L\}\vert}{L^2}=C_{\gamma_0}\cdot\mu_{\Thu}(\{\lambda\in\CM\CL(\Sigma)\vert\iota(\lambda_{\Sigma},\lambda)\le 1\}).$$
The claim follows since $\iota(\lambda_\Sigma,\cdot)=\ell_\Sigma(\cdot)$.
\end{proof}

In fact, using either the arguments in the proof of Theorem \ref{sat2}, or the combination of Theorem \ref{sat2} and Corollary \ref{kor-ratio}, note that the claim of Theorem \ref{sat2} holds true in much more generality:

\begin{kor}
Let $\Sigma$ be a complete hyperbolic surface of finite volume and homeomorphic to a once punctured torus and let $\gamma_0\subset\Sigma$ be a multicurve. The limit \eqref{eq101} exists and moreover we have
$$\lim_{L\to\infty}\frac{\vert\{\gamma\in\CS_{\gamma_0}\vert\iota(\lambda_0,\gamma)\le L\}\vert}{L^2}=C_{\gamma_0}\cdot\mu_{\Thu}(\{\lambda\in\CM\CL(\Sigma)\vert\iota(\lambda_0,\lambda)\le 1\})$$
for every filling current $\lambda_0\in\CC(\Sigma)$.\qed
\end{kor}

It follows in particular that Theorem \ref{sat2} also holds for instance if we replace hyperbolic length by length with respect to a metric with pinched negative curvature.

It remains to prove Theorem \ref{thm-density} which we devote the next section to.

\subsection{Densities of $\SL_2\BN$-invariant sets in $\BN^2$}
The semigroup $\SL_2\BN$ is the free semigroup generated by the two matrices
\begin{equation}\label{eq-generators}
\left(
\begin{array}{cc}
1 & 1 \\
0  & 1   
\end{array}
\right),\left(
\begin{array}{cc}
 1 & 0 \\
1 & 1   
\end{array}
\right).
\end{equation}
In this section we consider the action of $\SL_2\BN$ on $\BN^2\subset(0,\infty)^2$. Note that this action is free. Note also that the inverses of the matrices in \eqref{eq-generators} are the matrices 
$$\left(
\begin{array}{cc}
1 & -1 \\
0  & 1   
\end{array}
\right),\left(
\begin{array}{cc}
 1 & 0 \\
-1 & 1   
\end{array}
\right).$$
It follows that $\left(\begin{array}{c}a \\ b  \end{array}\right)\in\BN^2$ belongs to the $\SL_2\BN$-orbit of $\left(\begin{array}{c}a_0 \\ b_0  \end{array}\right)$ if and only if while running the euclidean algorithm step by step beginning with $\left(\begin{array}{c}a \\ b  \end{array}\right)$ one passes by $\left(\begin{array}{c}a_0 \\ b_0  \end{array}\right)$. In particular, the $\SL_2\BN$-orbit of $\left(\begin{array}{c}1 \\ 1  \end{array}\right)$ is exactly the set of positive integer vectors whose entries are prime to each other:
\begin{equation}\label{eq-density-primes0}
\SL_2\BN\left(\begin{array}{c}1 \\ 1  \end{array}\right)=\left\{\left(\begin{array}{c}a \\ b  \end{array}\right)\in\BN^2\middle\vert \gcd(a,b)=1\right\}.
\end{equation}
It is well-known that this set has a density:
\begin{equation}\label{eq-density-primes}
\lim_{L\to\infty}\frac 1{L^2}\left\vert\SL_2\BN\left(\begin{array}{c}1 \\ 1  \end{array}\right)\cap L\cdot U\right\vert=\frac 6{\pi^2}\vol(U\cap(0,\infty)^2)
\end{equation}
for every set $U\subset\BR^2$ bounded by a rectifiable Jordan curve. Here $\vol$ stands for the standard volume in $\BR^2$ and $\frac {\pi^2}6=\zeta(2)$ is famously the value of the Riemann zeta function $\zeta(s)=\sum_{n=1}^\infty n^{-s}$ at $s=2$.

As a first step towards the proof of Theorem \ref{thm-density} we prove a generalization of \eqref{eq-density-primes}.

\begin{prop}\label{density-orbits}
The set $S_{p,q}=\SL_2\BN\left(\begin{array}{c}p \\ q  \end{array}\right)$ has a density for all $p,q\in\BN$. In fact, 
\begin{equation}\label{eq-density-orbits-1}
\lim_{L\to\infty}\frac 1{L^2}\vert S_{p,q}\cap L\cdot U\vert=\frac 6{\pi^2pq}\vol(U\cap(0,\infty)^2)
\end{equation}
for every set $U\subset\BR^2$ bounded by a rectifiable Jordan curve. Moreover, we have that
$$\left\vert\left\{\left(\begin{array}{c}x \\ y  \end{array}\right)\in S_{p,q}\middle\vert x+y\le L\right\}\right\vert\le\frac{L^2}{2pq}$$
for all $L$.
\end{prop}
\begin{proof}
We deduce the existence of the density from a beautiful theorem of Maucourant \cite{pakito}. Consider $\SL_2\BN$ as a subset of the vector space $M_{2,2}(\BR)$ of $2$-by-$2$ real matrices and recall that $\SL_2\BN$ is exactly the intersection of $\SL_2\BZ$ with the set of matrices in $M_{2,2}(\BR)$ all of whose entries are non-negative. For all $L>0$, consider the measure
$$\nu_L=\frac 1{L^2}\sum_{A\in\SL_2\BN}\delta_{\frac 1L A}$$
on $M_{2,2}(\BR)$ where $\delta_x$ is the Dirac probability measure centered at $x$. From \cite{pakito} we obtain that when $L$ tends to $\infty$ the sequence of measures $\nu_L$ converges to some measure $\nu$ on $M_{2,2}(\BR)$. Although we will not need this fact, we remark that the measure $\nu$ is given explicitly in \cite[p. 361]{pakito}. 

Anyways, consider the map 
$$P:M_{2,2}(\BR)\to\BR^2,\ \ A\mapsto A\left(\begin{array}{c}p \\ q  \end{array}\right)$$
and notice that for all $U\subset\BR^2$ we have that
\begin{equation}\label{thesecatsareapain1}
\frac 1{L^2}\vert S_{p,q}\cap L\cdot U\vert=(P_*\nu_L)(U)
\end{equation}
where $P_*$ is the push-forward of the measure $\nu_L$ under $P$. It follows that
\begin{equation}\label{thesecatsareapain}
\lim_{L\to\infty}\frac 1{L^2}\vert S_{p,q}\cap L\cdot U\vert=(P_*\nu)(U)
\end{equation}
for any open set $U$ with $(P_*\nu)(\bar U\setminus U)=0$. Moreover, from \eqref{thesecatsareapain1} and \eqref{thesecatsareapain} we get that $P_*\nu(U)$ is bounded from above by $\vol(U)$. It follows that $P_*\nu$ is absolutely continuous with respect to the Lebesgue measure and hence that \eqref{thesecatsareapain} holds for all sets bounded by rectifiable Jordan curves. 

Now, the measure $P_*\nu$ is by construction invariant under the action of $\SL_2\BN$ on $(0,\infty)^2$. More concretely, this means that for any $A\in\SL_2\BN$ and $U\subset(0,\infty)^2$ we have that $P_*\nu(AU)=P_*\nu(U)$. Moreover, the measure $P_*\nu$ has the same scaling behavior as Lebesgue measure, meaning that for $U$ as above and $L\in\BR^+$ one has $P_*\nu(L\cdot U)=L^2\cdot P_*\nu(U)$. Up to scaling, Lebesgue measure is the only measure in the Lebesgue class with this behavior. Hence there is a constant $c$ with $P_*\nu(\cdot)=c\cdot\vol(\cdot)$ and thus
$$\lim_{L\to\infty}\frac 1{L^2}\vert S_{p,q}\cap L\cdot U\vert=c\cdot\vol(U\cap(0,\infty)^2)$$
for every $U$ bounded by a rectifiable Jordan curve. To conclude the proof of \eqref{eq-density-orbits-1} we need to show that $c=\frac 6{\pi^2pq}$. To see this, let $a,b>0$ and consider the triangle 
$$\Delta_{a,b}=\left\{\left(\begin{array}{c}x \\ y  \end{array}\right)\in(0,\infty)^2\middle\vert (a,b)\cdot\left(\begin{array}{c}x \\ y  \end{array}\right)\le 1\right\}.$$
Note that 
$$A\left(\begin{array}{c}p \\ q  \end{array}\right)\in L\cdot\Delta_{1,1}\Leftrightarrow A^t\left(\begin{array}{c}1 \\ 1  \end{array}\right)\in L\cdot\Delta_{p,q}$$
where $A\in\SL_2\BN$ and $A^t$ is the transpose of $A$. Since $\SL_2\BN$ is invariant under taking transposes we deduce that 
\begin{multline}\label{eq-trick}
\left\vert\left\{A\left(\begin{array}{c}p \\ q  \end{array}\right)\in L\cdot\Delta_{1,1}\middle\vert A\in\SL_2\BN\right\}\right\vert=\\
\left\vert\left\{A\left(\begin{array}{c}1 \\ 1  \end{array}\right)\in L\cdot\Delta_{p,q}\middle\vert A\in\SL_2\BN\right\}\right\vert
\end{multline}
for all $L>0$. From \eqref{eq-density-primes0} and \eqref{eq-density-primes} we obtain that 
$$\lim_{L\to\infty}\frac 1{L^2}\left\vert\left\{A\left(\begin{array}{c}1 \\ 1  \end{array}\right)\in L\cdot\Delta_{p,q}\middle\vert A\in\SL_2\BN\right\}\right\vert=\frac 6{\pi^2}\vol(\Delta_{p,q})=\frac 12\frac 6{\pi^2pq}$$
and when combining this with \eqref{eq-trick} it follows that
$$\lim_{L\to\infty}\frac 1{L^2}\left\vert\left\{A\left(\begin{array}{c}p \\ q  \end{array}\right)\in L\cdot\Delta_{1,1}\middle\vert A\in\SL_2\BN\right\}\right\vert=\frac 6{\pi^2pq}\vol(\Delta_{1,1})$$
which proves that $c=\frac 6{\pi^2pq}$, as we wanted. We have established \eqref{eq-density-orbits-1}.

The final claim in Proposition \ref{density-orbits} follows from \eqref{eq-trick} together with the observation that $(L\cdot \Delta_{p,q})\cap\BN^2$ has at most cardinality $\vol(L\cdot \Delta_{p,q})=\frac {L^2}{2pq}$.
\end{proof}

Armed with Proposition \ref{density-orbits} we are ready to prove Theorem \ref{thm-density}.

\begin{proof}[Proof of Theorem \ref{thm-density}]
Recall that the semigroup $\SL_2\BN$ acts freely on $\BN^2$. Moreover, because $\SL_2\BN$ is a free semigroup, two orbits $S_{p,q}=\SL_2\BN\left(\begin{array}{c}p \\ q  \end{array}\right)$ and $S_{r,s}=\SL_2\BN\left(\begin{array}{c}r \\ s  \end{array}\right)$ intersect if and only if one orbit is contained in the other. It follows that any invariant set $\CI$ can be written in a unique way as the disjoint union of a set $\CO$ of orbits
$$\CI=\bigsqcup_{(p,q)^t\in\CO}S_{p,q}$$
where the superscript ${ }^t$ denotes, as always, the transpose. Note that the infinite sum
$$\lambda=\sum_{(p,q)^t\in\CO}\frac 6{\pi^2pq}$$
exists because it is bounded from above by $1$ and all its summands are positive. We claim that for any $U\subset(0,\infty)^2$ with rectifiable boundary we have
$$\lim_{L\to\infty}\frac 1{L^2}\vert\CI\cap L\cdot U\vert=\lambda\vol(U).$$ 
Since the orbits $S_{p,q}$ are pairwise disjoint, the claim follows directly from Proposition \ref{density-orbits} if the set $\CO$ is finite. In particular, one has that
\begin{equation}\label{eq-liminf}
\liminf_{L\to\infty}\frac 1{L^2}\vert\CI\cap L\cdot U\vert\ge\left(
\sum_{(p,q)\in\CO'}\frac 6{\pi^2pq}
\right)\cdot\vol(U)
\end{equation}
for any finite subset $\CO'\subset\CO$. Note that for any $\epsilon$ we can find a finite set $\CO'$ such that if $\CP=\CO\setminus\CO'$ is its complement we have 
\begin{equation}\label{guivarchjustsneezed}
\sum_{(p,q)^t\in\CP}\frac 6{\pi^2pq}\le\epsilon.
\end{equation}
The claim will follow then from \eqref{eq-liminf} if we show that for any $U\subset(0,1)^2$ as above and set $\CP\subset\CO$ satisfying \eqref{guivarchjustsneezed} we have
$$\limsup_{L\to\infty}\frac 1{L^2}\left\vert\left(\bigsqcup_{(p,q)^t\in\CP}S_{p,q}\right)\cap L\cdot U\right\vert\le c_U\cdot\epsilon$$
for some constant $c_U$ depending on $U$. In fact, letting $\ell_U$ be such that $U\subset\ell_U\cdot\Delta_{1,1}$ we will prove that $c_U=\ell_U^2$ does the trick -- here $\Delta_{1,1}$ is as in the proof of Proposition \ref{density-orbits}. For any such $U$ we get that
\begin{align*}
\left\vert\left(\bigsqcup_{(p,q)^t\in\CP}S_{p,q}\right)\cap L\cdot U\right\vert
  &\le\left\vert\left(\bigsqcup_{(p,q)^t\in\CP}S_{p,q}\right)\cap L\cdot\ell_U\cdot\Delta_{1,1}\right\vert\\
  &=\left\vert\left(\bigsqcup_{(p,q)^t\in\CP}S_{p,q}\right)\cap\left\{\left(\begin{array}{c}x \\ y  \end{array}\right)\middle\vert x+y\le L\ell_U\right\}\right\vert\\
  &=\sum_{(p,q)^t\in\CP}\left\vert S_{p,q}\cap\left\{\left(\begin{array}{c}x \\ y  \end{array}\right)\middle\vert x+y\le L\ell_U\right\}\right\vert\\
  &\le\sum_{(p,q)^t\in\CP}\frac{L^2\ell_U^2}{2pq}< L^2c_U\sum_{(p,q)^t\in\CP}\frac 6{\pi^2pq}\le L^2c_U\epsilon
 \end{align*}
Having established the bound we needed, we have proved Theorem \ref{thm-density}.
\end{proof}

The reader might wonder if one really needed to say anything justifying that the density of a disjoint union of set is the sum of the densities of the individual sets. Well, this statement  is in fact not true in general: in general every countable set is the countable union of singletons, which obviously have density $0$. In other words, the last claim of Proposition \ref{density-orbits} actually plays a central role in the proof of Theorem \ref{thm-density}.


\end{document}